\documentclass{amsart}

\usepackage{amsmath}
\usepackage{amsthm}
\usepackage{amsfonts}
\usepackage{dsfont}
\usepackage{mathabx}

\newcommand{\reals}{\mathbb{R}}

\newcommand{\bracketb}[1]{\Big[#1\Big]}
\newcommand{\bracketc}[1]{\bigg[#1\bigg]}

\newcommand{\pb}[1]{\left\{#1\right\}}
\newcommand{\pba}[1]{\big\{#1\big\}}

\newcommand{\coma}[1]{\big[#1\big]}

\newcommand{\Ccom}[3]{\big[[#1,#2],#3\big]}
\newcommand{\norm}[1]{\left|\left|#1\right|\right|}
\newcommand{\fatcom}[1]{\ldbrack{#1}\rdbrack}
\newcommand{\abs}[1]{\left|#1\right|}

\newcommand{\para}[1]{\left(#1\right)}
\newcommand{\paraa}[1]{\big(#1\big)}
\newcommand{\parab}[1]{\Big(#1\Big)}
\newcommand{\parac}[1]{\bigg(#1\bigg)}

\newcommand{\diag}{\operatorname{diag}}

\newtheorem{theorem}{Theorem}[section]

\newtheorem{lemma}[theorem]{Lemma}
\newtheorem{proposition}[theorem]{Proposition}
\newtheorem{example}[theorem]{Example}
\theoremstyle{definition}
\newtheorem{definition}[theorem]{Definition}
\theoremstyle{remark}
\newtheorem{remark}[theorem]{Remark}
\numberwithin{equation}{section}

\newcommand{\tr}{\operatorname{tr}}
\newcommand{\Tr}{\operatorname{Tr}}
\newcommand{\Z}{\mathcal{Z}}

\newcommand{\B}{\mathcal{B}}
\renewcommand{\P}{\mathcal{P}}
\renewcommand{\S}{\mathcal{S}}
\newcommand{\T}{\mathcal{T}}
\newcommand{\C}{\mathcal{C}}
\newcommand{\W}{\mathcal{W}}
\newcommand{\R}{\mathcal{R}}
\newcommand{\J}{\mathcal{J}}
\newcommand{\JM}{\mathcal{J}_M}
\newcommand{\SA}{\S_A}

\newcommand{\xv}{\vec{x}}
\newcommand{\yv}{\vec{y}}
\newcommand{\gb}{\,\bar{\!g}}
\renewcommand{\d}{\partial}
\newcommand{\TSigma}{T\Sigma}
\newcommand{\eps}{\varepsilon}

\newcommand{\nablab}{\bar{\nabla}}

\newcommand{\Gammab}{\bar{\Gamma}}
\newcommand{\Rb}{\bar{R}}
\newcommand{\TN}{T^{(N)}}

\newcommand{\vphi}{\varphi}
\newcommand{\Nh}{\hat{N}}

\newcommand{\av}{\vec{a}}
\newcommand{\cv}{\vec{c}}
\newcommand{\qv}{\vec{q}}
\newcommand{\nAv}{\vec{n}_A}
\renewcommand{\mid}{\mathds{1}}
 
\newcommand{\D}{\mathcal{D}}

\newcommand{\uh}{\hat{u}}
\newcommand{\uha}{\uh_\alpha}
\newcommand{\vh}{\hat{v}}
\newcommand{\vha}{\vh_\alpha}
\newcommand{\Xa}{X_\alpha}
\newcommand{\ha}{\hbar_\alpha}
\newcommand{\hbara}{\ha}

\newcommand{\Na}{N_\alpha}
\newcommand{\Ta}{T_\alpha}
\newcommand{\Sa}{S_\alpha}

\newcommand{\limainfty}{\lim_{\alpha\to\infty}}
\newcommand{\limai}{\limainfty}
\newcommand{\la}{\lambda_\alpha}
\newcommand{\Kh}{\hat{K}}
\newcommand{\Kha}{\hat{K}_\alpha}
\newcommand{\Dh}{\hat{D}}
\newcommand{\Dha}{\Dh_\alpha}
\newcommand{\fh}{\hat{f}}
\newcommand{\fha}{\fh_\alpha}
\newcommand{\limNinf}{\lim_{N\to\infty}}

\newcommand{\pv}{\vec{p}}
\newcommand{\dv}{\vec{d}}
\newcommand{\hh}{\hat{h}}
\newcommand{\hha}{\hh_\alpha}
\newcommand{\Deltah}{\hat{\Delta}}
\newcommand{\Deltaha}{\Deltah_\alpha}
\renewcommand{\dh}{\,\hat{\!\partial}}
\newcommand{\dha}{\dh_\alpha}
\newcommand{\NAa}{N_{A\alpha}}
\newcommand{\thetah}{\hat{\theta}}

\newcommand{\nv}{\vec{n}}
\newcommand{\ff}{f\!f}
\newcommand{\Wd}{W^\dagger}
\newcommand{\Ch}{\hat{C}}
\newcommand{\Gammah}{\hat{\Gamma}}
\newcommand{\gammah}{\hat{\gamma}}
\newcommand{\gammaha}{\gammah_\alpha}
\newcommand{\chih}{\hat{\chi}}

\newcommand{\Gh}{\hat{G}}
\newcommand{\Ph}{\hat{\P}}
\newcommand{\ShA}{\hat{\S}_A}
\newcommand{\trh}{\widehat{\operatorname{tr}}\,}
\newcommand{\mv}{\vec{m}}
\newcommand{\vphiv}{\vec{\vphi}}
\newcommand{\gd}{g^\dagger}
\newcommand{\hd}{h^\dagger}

\title[]{Multi linear formulation of differential geometry and matrix regularizations}

\author{Joakim Arnlind}
\address[Joakim Arnlind]{Max Planck Institute for Gravitational Physics\\ 
Am M\"uhlenberg 1\\
D-14476 Golm\\
Germany}
\email{joakim.arnlind@aei.mpg.de}

\author{Jens Hoppe}
\address[Jens Hoppe]{Department of Mathematics\\
KTH\\
S-10044 Stockholm\\
Sweden}
\email{hoppe@math.kth.se}

\author{Gerhard Huisken}
\address[Gerhard Huisken]{Max Planck Institute for Gravitational Physics\\
Am M\"uhlenberg 1\\
D-14476 Golm\\
Germany}
\email{gerhard.huisken@aei.mpg.de}

\thanks{}

\subjclass[2000]{}

\begin{document}

\begin{abstract}
  We prove that many aspects of the differential geometry of embedded
  Riemannian manifolds can be formulated in terms of multi linear
  algebraic structures on the space of smooth functions.  In particular,
  we find algebraic expressions for Weingarten's formula, the Ricci
  curvature and the Codazzi-Mainardi equations.

  For matrix analogues of embedded surfaces we define discrete
  curvatures and Euler characteristics, and a non-commutative
  Gauss--Bonnet theorem is shown to follow.  We derive simple
  expressions for the discrete Gauss curvature in terms of matrices
  representing the embedding coordinates, and a large class of
  explicit examples is provided. Furthermore, we illustrate the fact
  that techniques from differential geometry can carry over to
  matrix analogues by proving that a bound on the discrete Gauss
  curvature implies a bound on the eigenvalues of the discrete Laplace
  operator.
\end{abstract}

\maketitle

\tableofcontents

\section{Introduction}

\noindent It is generally interesting to study in what ways
information about the geometry of a differentiable manifold $\Sigma$
can be extracted as algebraic properties of the algebra of smooth
functions $C^\infty(\Sigma)$. In case $\Sigma$ is a Poisson manifold,
this algebra has a second (apart from the commutative multiplication
of functions) bilinear (non-associative) algebra structure, the
Poisson bracket. The bracket is compatible with the commutative
multiplication via Leibniz rule, thus carrying the basic properties of
a derivation.

On a surface $\Sigma$, with local coordinates $u^1$ and $u^2$, one can
define
\begin{align*}
  \{f,h\} = \frac{1}{\sqrt{g}}\parac{\frac{\d f}{\d u^1}\frac{\d h}{\d u^2}-
  \frac{\d h}{\d u^1}\frac{\d f}{\d u^2}},
\end{align*}
where $g$ is the determinant of the induced metric tensor, and one
readily checks that $\paraa{C^\infty(\Sigma),\{\cdot,\cdot\}}$ is a
Poisson algebra. Having only this very particular combination of
derivatives at hand, it seems at first unlikely that one can encode
geometric information of $\Sigma$ in Poisson algebraic
expressions. Surprisingly, it turns out that many differential
geometric quantities can be computed in a completely algebraic way,
cp.  Theorem \ref{thm:ricciCurvature} and Theorem
\ref{thm:CMNambu}. For instance, the Gaussian curvature of a surface
embedded in $\reals^m$ can be written as
\begin{align}\label{eq:introKpb}
  K=\sum_{j,k,l=1}^m\parac{\frac{1}{2}\{\{x^j,x^k\},x^k\}\{\{x^j,x^l\},x^l\}
    -\frac{1}{4}\{\{x^j,x^k\},x^l\}\{\{x^j,x^k\},x^l\}},
\end{align}
where $x^i(u^1,u^2)$ are the embedding coordinates of the surface.

For a general $n$-dimensional manifold $\Sigma$, we are led to
consider Nambu brackets \cite{n:generalizedmech}, i.e. multi-linear
alternating $n$-ary maps from $C^\infty(\Sigma)\times\cdots\times
C^\infty(\Sigma)$ to $C^\infty(\Sigma)$, defined by
\begin{align*}
  \{f_1,\ldots,f_n\} = \frac{1}{\sqrt{g}}\eps^{a_1\cdots a_n}\paraa{\d_{a_1}f_1}\cdots\paraa{\d_{a_n} f_n}.
\end{align*}
% Just as Gauss originally defined $K$ as the (signed) ratio of
% infinitesimal areas swept out by the normal, resp. on the surface, the
% product of the principle curvatures of an $n$-dimensional Euclidean
% hypersurface (with normal $\nv$) being a (signed) ratio of volumes,
% can be written as
% \begin{align*}
%   \det W = \frac{1}{n!}\eps^{i_1,\ldots,i_n i}\{n_{i_1},\ldots,n_{i_n}\}n_i.
% \end{align*}
In the case of surfaces, our initial motivation for studying the
problem came from matrix regularizations of Membrane Theory. Classical
solutions in Membrane Theory are 3-manifolds with vanishing mean
curvature in $\reals^{1,d}$. Considering one of the coordinates to be
time, the problem can also be formulated in a dynamical way as
surfaces sweeping out volumes of vanishing mean curvature. In this
context, a regularization was introduced replacing the infinite
dimensional function algebra on the surface by an algebra of $N\times
N$ matrices \cite{h:phdthesis}. If we let $\Ta$ be a linear map from
smooth functions to hermitian $\Na\times \Na$ matrices, the main
properties of the regularization are
\begin{align*}
  &\limainfty\norm{\Ta(f)\Ta(g)-\Ta(fg)}=0,\\
  &\limainfty\norm{\frac{1}{i\hbara}[\Ta(f),\Ta(h)]-\Ta(\{f,h\})}=0,
\end{align*}
where $\hbara$ is a real valued function tending to zero as
$\Na\to\infty$
%where $||\cdot||$ denotes the operator norm 
(see Section \ref{sec:matrixRegularizations} for details), and
therefore it is natural to regularize the system by replacing
(commutative) multiplication of functions by (non-commutative)
multiplication of matrices and Poisson brackets of functions by
commutators of matrices.

Although we may very well consider $\Ta(\frac{\d f}{\d u^1})$, its
relation to $\Ta(f)$ is in general not simple. However, the particular
combination of derivatives in $\Ta(\{f,h\})$ is expressed in terms of
a commutator of $\Ta(f)$ and $\Ta(h)$. In the context of Membrane
Theory, it is desirable to have geometrical quantities in a form that
can easily be regularized, which is the case for any expression
constructed out of multiplications and Poisson brackets.  For
instance, solving the equations of motion for the regularized membrane
gives sequences of matrices that correspond to the embedding
coordinates of the surface. Since the set of solutions contains
regularizations of surfaces of arbitrary topology, one would like to
be able to compute the genus corresponding to particular
solutions. The regularized form of (\ref{eq:introKpb}) provides a way
of resolving this problem.

The paper is organized as follows: In Section \ref{sec:preliminaries}
we introduce the relevant notation by recalling some basic facts about
submanifolds. In Section \ref{sec:nambuPoissonFormulation} we
formulate several basic differential geometric objects in terms of
Nambu brackets, and in Section \ref{sec:normalVectors} we provide a
construction of a set of orthonormal basis vectors of the normal
space. Section \ref{sec:CodazziMainardi} is devoted to the study of
the Codazzi-Mainardi equations and how one can rewrite them in terms
of Nambu brackets. In Section \ref{sec:surfaces} we study the
particular case of surfaces, for which many of the introduced formulas
and concepts are particularly nice and in which case one can construct
the complex structure in terms of Poisson brackets.

In the second part of the paper, starting with Section
\ref{sec:matrixRegularizations}, we study the implications of our
results for matrix regularizations of compact surfaces. In particular,
a discrete version of the Gauss-Bonnet theorem is derived in Section
\ref{sec:discreteGB} and a proof that the discrete Gauss curvature
bounds the eigenvalues of the discrete Laplacian is found in Section
\ref{sec:laplaceBound}.

\section{Preliminaries}\label{sec:preliminaries}

\noindent To introduce the relevant notations, we shall recall some
basic facts about submanifolds, in particular Gauss' and Weingarten's
equations (see
e.g. \cite{kn:foundationsDiffGeometryI,kn:foundationsDiffGeometryII}
for details). For $n\geq 2$, let $\Sigma$ be a $n$-dimensional manifold embedded in a
Riemannian manifold $M$ with $\dim M=n+p\equiv m$. Local coordinates on $M$ will be denoted by
$x^1,\ldots,x^m$, local coordinates on $\Sigma$ by $u^1,\ldots,u^n$,
and we regard $x^1,\ldots,x^m$ as being functions of $u^1,\ldots,u^n$
providing the embedding of $\Sigma$ in $M$. The metric tensor on $M$
is denoted by $\gb_{ij}$ and the induced metric on $\Sigma$ by
$g_{ab}$; indices $i,j,k,l,n$ run from $1$ to $m$, indices
$a,b,c,d,p,q$ run from $1$ to $n$ and indices $A,B,C,D$ run from $1$
to $p$. Furthermore, the covariant derivative and the Christoffel
symbols in $M$ will be denoted by $\nablab$ and $\Gammab^{i}_{jk}$
respectively.

The tangent space $\TSigma$ is regarded as a subspace of the tangent
space $TM$ and at each point of $\Sigma$ one can choose
$e_a=(\d_ax^i)\d_i$ as basis vectors in $\TSigma$, and in this basis
we define $g_{ab}=\gb(e_a,e_b)$. Moreover, we choose a set of normal
vectors $N_A$, for $A=1,\ldots,p$, such that
$\gb(N_A,N_B)=\delta_{AB}$ and $\gb(N_A,e_a)=0$.

The formulas of Gauss and Weingarten split the covariant derivative in
$M$ into tangential and normal components as
\begin{align}
  &\nablab_X Y = \nabla_X Y + \alpha(X,Y)\label{eq:GaussFormula}\\
  &\nablab_XN_A = -W_A(X) + D_XN_A\label{eq:WeingartenFormula}
\end{align}
where $X,Y\in \TSigma$ and $\nabla_X Y$, $W_A(X)\in\TSigma$ and
$\alpha(X,Y)$, $D_XN_A\in\TSigma^\perp$. By expanding $\alpha(X,Y)$ in
the basis $\{N_1,\ldots,N_p\}$ one can write (\ref{eq:GaussFormula}) as
\begin{align}
  &\nablab_X Y = \nabla_X Y + \sum_{A=1}^ph_A(X,Y)N_A,\label{eq:GaussFormulah}
\end{align}
and we set $h_{A,ab} = h_A(e_a,e_b)$. From the above equations one derives the relation
\begin{align}
  h_{A,ab} &= -\gb\paraa{e_a,\nablab_b N_A},
\end{align}
as well as Weingarten's equation
\begin{align}
  h_A(X,Y) = \gb\paraa{W_A(X),Y},  
\end{align}
which implies that $(W_A)^a_b = g^{ac}h_{A,cb}$, where $g^{ab}$
denotes the inverse of $g_{ab}$.  

From formulas (\ref{eq:GaussFormula}) and (\ref{eq:WeingartenFormula})
one obtains Gauss' equation, i.e. an expression for the curvature $R$
of $\Sigma$ in terms of the curvature $\Rb$ of $M$, as
\begin{equation}\label{eq:GaussEquation}
  \begin{split}
    g\paraa{R(X,Y)Z,V} =
    \gb&\paraa{\Rb(X,Y)Z,V}-\gb\paraa{\alpha(X,Z),\alpha(Y,V)}\\
    &+\gb\paraa{\alpha(Y,Z),\alpha(X,V)},    
  \end{split}
\end{equation}
where $X,Y,Z,V\in\TSigma$. As we shall later on consider the Ricci curvature,
let us note that (\ref{eq:GaussEquation}) implies
\begin{align}
  \R^p_b = g^{pd}g^{ac}\gb\paraa{\Rb(e_c,e_d)e_b,e_a}
  +\sum_{A=1}^p\bracketb{(W_A)^a_a(W_A)_b^p-(W_A^2)_b^p}
\end{align}
where $\R$ is the Ricci curvature of $\Sigma$ considered as a map
$\TSigma\to\TSigma$.  We also recall the mean curvature vector,
defined as
\begin{align}
  H = \frac{1}{n}\sum_{A=1}^p\paraa{\tr W_A}N_A.
\end{align}

\section{Nambu bracket formulation}\label{sec:nambuPoissonFormulation}

\noindent In this section we will prove that one can express many
aspects of the differential geometry of an embedded manifold $\Sigma$
in terms of a Nambu bracket introduced on $C^\infty(\Sigma)$.
Let $\rho:\Sigma\to\reals$ be an arbitrary non-vanishing
density and define
\begin{align}\label{eq:PbracketDef}
  \{f_1,\ldots,f_n\} = \frac{1}{\rho}\eps^{a_1\cdots a_n}\paraa{\d_{a_1}f_1}\cdots\paraa{\d_{a_n} f_n}
\end{align}
for all $f_1,\ldots,f_n\in C^\infty(\Sigma)$, where $\eps^{a_1\cdots
  a_n}$ is the totally antisymmetric Levi-Civita symbol with
$\eps^{12\cdots n}=1$. Together with this multi-linear map, $\Sigma$ is
a Nambu-Poisson manifold.

The above Nambu bracket arises from the choice of a
volume form on $\Sigma$. Namely, let $\omega$ be a volume form and
define $\{f_1,\ldots,f_n\}$ via the formula
\begin{align}\label{eq:nambuVolumeForm}
 \{f_1,\ldots,f_n\}\omega = df_1\wedge\cdots\wedge df_n.
\end{align}
Writing $\omega=\rho\, du^1\wedge\cdots\wedge du^n$ in local
coordinates, and evaluating both sides of (\ref{eq:nambuVolumeForm})
on the tangent vectors $\d_{u^1},\ldots,\d_{u^n}$ gives
\begin{align*}
  \{f_1,\ldots,f_n\} = \frac{1}{\rho}\det\parac{\frac{\d(f_1,\ldots,f_n)}{\d(u^1,\ldots,u^n)}}
  =\frac{1}{\rho}\eps^{a_1\cdots a_n}\paraa{\d_{a_1}f_1}\cdots\paraa{\d_{a_n}f_n}.
\end{align*}
To define the objects which we will consider, it is convenient to
introduce some notation. Let
$x^1(u^1,\ldots,u^n),\ldots,x^m(u^1,\ldots,u^n)$ be the embedding
coordinates of $\Sigma$ into $M$, and let $n_A^i(u^1,\ldots,u^n)$ denote the
components of the orthonormal vectors $N_A$, normal to $\TSigma$. Using
multi-indices $I=i_1\cdots i_{n-1}$ and $\av=a_1\cdots a_{n-1}$ we define
\begin{align*}
  &\{f,\xv^I\} \equiv \{f,x^{i_1},x^{i_2},\ldots,x^{i_{n-1}}\}\\
  &\{f,\nAv^I\} \equiv \{f,n_A^{i_1},n_A^{i_2},\ldots,n_A^{i_{n-1}}\},
\end{align*}
together with
\begin{align*}
  &\d_{\av}\xv^I \equiv \paraa{\d_{a_1}x^{i_1}}\paraa{\d_{a_2}x^{i_2}}\cdots\paraa{\d_{a_{n-1}}x^{i_{n-1}}}\\
  &\paraa{\nablab_{\av}\nAv}^I \equiv \paraa{\nablab_{a_1}N_A}^{i_1}\paraa{\nablab_{a_2}N_A}^{i_2}\cdots\paraa{\nablab_{a_{n-1}}N_A}^{i_{n-1}}\\
  &\gb_{IJ} \equiv \gb_{i_1j_1}\gb_{i_2j_2}\cdots\gb_{i_{n-1}j_{n-1}}\\
  &g_{\av\cv}\equiv g_{a_1c_1}g_{a_2c_2}\cdots g_{a_{n-1}c_{n-1}}.
\end{align*}
We now introduce the main objects of our study
\begin{align}
  \P^{iJ} &= \frac{1}{\sqrt{(n-1)!}}\{x^i,\xv^J\} = \frac{1}{\sqrt{(n-1)!}}\frac{\eps^{a\av}}{\rho}\paraa{\d_ax^i}\paraa{\d_{\av}\xv^J}\\
  % \S_A^{Ii} &= \{\xv^I,n_A^i\} = \frac{1}{\rho}\eps^{\av b}\paraa{\d_{\av}\xv^I}\paraa{\nablab_bN_A}^j\\
  % \T_A^{iI}&=\{x^i,\nAv^I\} = \frac{1}{\rho}\eps^{a\av}\paraa{\d_ax^i}\paraa{\nablab_{\av}\nAv}^I
  \S_A^{iJ}&=\frac{(-1)^n}{\sqrt{(n-1)!}}\frac{\eps^{a\av}}{\rho}\paraa{\d_ax^i}\paraa{\nablab_{\av}\nAv}^J\\
  \T_A^{Ij} &=\frac{(-1)^n}{\sqrt{(n-1)!}}\frac{\eps^{\av a}}{\rho}\paraa{\d_{\av}\xv^I}\paraa{\nablab_aN_A}^j
\end{align}
from which we construct
\begin{align}
  \paraa{\P^2}^{ik} &= \P^{iI}\P^{kJ}\gb_{IJ}\\
  % \paraa{\B_A}^{ik} &= \P^{iI}(\S_A)^{Jk}\gb_{IJ}\\
  % \paraa{\T_A\S_A}^{ik} &=(\T_A)^{iI}(\S_A)^{Jk}\gb_{IJ}.
  \paraa{\B_A}^{ik} &= \P^{iI}(\T_A)^{Jk}\gb_{IJ}\\
  \paraa{\S_A\T_A}^{ik} &=(\S_A)^{iI}(\T_A)^{Jk}\gb_{IJ}.
\end{align}
By lowering the second index with the metric $\gb$, we will also
consider $\P^2$, $\B_A$ and $\T_A\S_A$ as maps $TM\to TM$. Note that
both $\S_A$ and $\T_A$ can be written in terms of Nambu brackets, e.g.
\begin{align*}
  \T_A^{Ij} = \frac{(-1)^n}{\sqrt{(n-1)!}}\bracketb{\{\xv^I,n_A^j\}+\{\xv^I,x^k\}\Gammab^j_{kl}n_A^l}.
\end{align*}
Let us now investigate some properties of the maps defined above. As it will appear frequently, we define 
\begin{align}
  \gamma = \frac{\sqrt{g}}{\rho}.
\end{align}
It is useful to note that (cp. Proposition \ref{prop:TrPBST})
\begin{align*}
  \gamma^2 = \sum_{i,j,I,J=1}^m\frac{1}{n!}\gb_{ij}\{x^i,\xv^I\}\gb_{IJ}\{x^j,\xv^J\},
\end{align*}
and to recall the cofactor expansion of the inverse of a matrix:
\begin{lemma}
  Let $g^{ab}$ denote the inverse of $g_{ab}$ and $g=\det(g_{ab})$. Then
  \begin{align}
    gg^{ba} = \frac{1}{(n-1)!}\eps^{aa_1\cdots a_{n-1}}\eps^{bb_1\cdots b_{n-1}}g_{a_1b_1}g_{a_2b_2}\cdots g_{a_{n-1}b_{n-1}}.
  \end{align}
\end{lemma}

\begin{proposition}\label{prop:PBSTproperties}
  For $X\in TM$ it holds that
  \begin{align}
    &\P^2(X) = \gamma^2\gb(X,e_a)g^{ab}e_b\label{eq:P2X}\\
    &\B_A(X) = -\gamma^2\gb(X,\nablab_a N_A)g^{ab}e_b\label{eq:BAX}\\
    &\S_A\T_A(X) = \gamma^2(\det W_A)\gb(X,\nablab_aN_A)h_A^{ab}e_b,
  \end{align}
  and for $Y\in\TSigma$ one obtains
  \begin{align}
    &\P^2(Y) = \gamma^2Y\label{eq:P2Y}\\
    &\B_A(Y) = \gamma^2W_A(Y)\\
    &\S_A\T_A(Y) = -\gamma^2(\det W_A)Y.
  \end{align}
\end{proposition}

\begin{proof}
  Let us provide a proof for equations (\ref{eq:P2X}) and (\ref{eq:P2Y}); the other
  formulas can be proved analogously.
  \begin{align*}
    \P^2(X) &= \P^{iI}\P^{jJ}\gb_{IJ}\gb_{jk}X^k\d_i = 
    \frac{\eps^{a\av}\eps^{c\cv}}{\rho^2(n-1)!}\paraa{\d_ax^i}\paraa{\d_{\av}x^I}\paraa{\d_cx^j}\paraa{\d_{\cv}x^J}\gb_{IJ}\gb_{jk}X^k\d_i\\
    &= \frac{\eps^{a\av}\eps^{c\cv}}{\rho^2(n-1)!}g_{a_1c_1}\cdots g_{a_{n-1}c_{n-1}}\paraa{\d_ax^i}\paraa{\d_cx^j}\gb_{jk}X^k\d_i\\
    &= \gamma^2g^{ac}\paraa{\d_ax^i}\paraa{\d_cx^j}\gb_{jk}X^k\d_i
    = \gamma^2\gb(X,e_c)g^{ca}e_a.
  \end{align*}
  Choosing a tangent vector $Y=Y^ce_c$ gives immediately that $\P^2(Y)=\gamma^2Y$.
\end{proof}

\noindent For a map $\B:TM\to TM$ we denote the trace by $\Tr\B\equiv
\B^i_i$ and for a map $W:\TSigma\to\TSigma$ we denote the trace by $\tr W\equiv W^a_a$.

\begin{proposition}\label{prop:TrPBST}
  It holds that
  \begin{align}
    \frac{1}{n}\Tr\P^2 &= \gamma^2\\
    \Tr\B_A &= \gamma^2\tr W_A\\
    \frac{1}{n}\Tr\S_A\T_A &= -\gamma^2(\det W_A).
  \end{align}
\end{proposition}

\begin{remark}
  For a hypersurface (with normal $N=n^i\d_i$) in $\reals^{n+1}$,
  \begin{align}
    \det W &= (-1)^n\frac{\{x^{i_1},\ldots,x^{i_n}\}\{n_{i_1},\ldots,n_{i_n}\}}
    {\{x^{k_1},\ldots,x^{k_n}\}\{x_{k_1},\ldots,x_{k_n}\}}\\
    &=\frac{1}{\gamma n!}\eps^{i_1\cdots i_n i}\{n_{i_1},\ldots,n_{i_n}\}n_i,\notag
  \end{align}
  the signed ratio of infinitesimal volumes swept out on $S^n$ (by
  $N$), resp $\Sigma$ (which can easily be obtained directly by simply
  writing out the determinant of the second fundamental form,
  $h=\det(-\d_ax^i\d_bn_i)$); in fact, all the symmetric functions of
  the principal curvatures are related to ratios of products of two
  Nambu brackets (cp. the paragraph after Proposition
  \ref{prop:HyperSurfaceGaussian}). Namely, the $k$'th symmetric
  curvature is given by
  \begin{align}
    (-1)^k\frac{\{x^{i_1},\ldots,x^{i_n}\}
      \{n_{i_{1}},\ldots,n_{i_k},x_{i_{k+1}},\ldots,x_{i_n}\}}
    {\{x^{k_1},\ldots,x^{k_n}\}\{x_{k_1},\ldots,x_{k_n}\}}.
  \end{align}
\end{remark}

\noindent A direct consequence of Propositions
\ref{prop:PBSTproperties} and \ref{prop:TrPBST} is that one can write the projection onto
$\TSigma$, as well as the mean curvature vector,  in terms of Nambu brackets.
\begin{proposition}
  The map 
  \begin{align}
    \gamma^{-2}\P^2=\frac{n}{\Tr\P^2}\P^2:TM\to\TSigma
  \end{align}
  is the orthogonal projection of $TM$ onto $\TSigma$. Furthermore,
  the mean curvature vector can be written as
\begin{align*}
  H = \frac{1}{\Tr\P^2}\sum_{A=1}^p\paraa{\Tr\B_A}N_A.
\end{align*}
\end{proposition}

\noindent Proposition \ref{prop:PBSTproperties} tells us that $\gamma^{-2}\B_A$
equals the Weingarten map $W_A$, when restricted to $\TSigma$. What is
the geometrical meaning of $\B_A$ acting on a normal vector? It turns
out that the maps $\B_A$ also provide information about the covariant
derivative in the normal space. If one defines $(D_X)_{AB}$ through
\begin{align*}
  D_XN_A = \sum_{B=1}^p(D_X)_{AB}N_B
\end{align*}
for $X\in\TSigma$, then one can prove the following relation to
the maps $\B_A$.
\begin{proposition}
  For $X\in\TSigma$ it holds that
  \begin{align}
    \gb\paraa{\B_B(N_A),X}=\gamma^2\paraa{D_X}_{AB}.
  \end{align}
\end{proposition}

\begin{proof}
  For a vector $X=X^ae_a$, it follows from Weingarten's formula (\ref{eq:WeingartenFormula}) that
  \begin{align*}
    (D_X)_{AB} = \gb\paraa{\nablab_X N_A,N_B}. % = X^a\gb(\nablab_a N_A,N_B).
  \end{align*}
  On the other hand, with the formula from Proposition
  \ref{prop:PBSTproperties}, one computes
  \begin{align*}
     \gb\paraa{\B_B(N_A),X} &= -\gamma^2\gb\paraa{N_A,\nablab_aN_B}g^{ab}g_{bc}X^c
     = -\gamma^2\gb\paraa{N_A,\nablab_XN_B}\\
     &= -\gamma^2(D_X)_{BA}=\gamma^2(D_X)_{AB}.
  \end{align*}
  The last equality is due to the fact that $D$ is a covariant
  derivative, which implies that
  $0=D_X\gb(N_A,N_B)=\gb(D_XN_A,N_B)+\gb(N_A,D_XN_B)$.
\end{proof}

\noindent Thus, one can write Weingarten's formula as
\begin{align}
  \gamma^2\nablab_XN_A = -\B_A(X)+\sum_{B=1}^p\gb\paraa{\B_B(N_A),X}N_B,
\end{align}
and since $h_A(X,Y) = \gamma^{-2}\gb(\B_A(X),Y)$ Gauss' formula becomes
\begin{align}\label{eq:GaussformulaB}
  \nablab_XY = \nabla_XY+\frac{1}{\gamma^2}\sum_{A=1}^p\gb\paraa{\B_A(X),Y}N_A.
\end{align}
Let us now turn our attention to the curvature of $\Sigma$. Since
Nambu brackets involve sums over all vectors in the basis of
$\TSigma$, one can not expect to find expressions for quantities that
involve a choice of tangent plane, e.g. the sectional curvature
(unless $\Sigma$ is a surface). However, it turns out that one can
write the Ricci curvature as an expression involving Nambu brackets.
\begin{theorem}\label{thm:ricciCurvature}
  Let $\R$ be the Ricci curvature of $\Sigma$, considered as a map
  $\TSigma\to\TSigma$, and let $R$ denote the scalar curvature. For
  any $X\in\TSigma$ it holds that
  \begin{align}
    %g^{pd}g^{ac}\gb\paraa{\Rb(e_c,e_d)e_b,e_a}X^be_p
    &\R(X) = \frac{1}{\gamma^4}\paraa{\P^2}^{ik}\paraa{\P^2}^{lm}\Rb_{ijkl}X^j\d_m
    +\frac{1}{\gamma^4}\sum_{A=1}^p\bracketb{(\Tr\B_A)\B_A(X)-\B_A^2(X)}\\
    &R = \frac{1}{\gamma^4}\paraa{\P^2}^{ik}\paraa{\P^2}^{jl}\Rb_{ijkl}
    +\frac{1}{\gamma^4}\sum_{A=1}^p\bracketb{(\Tr\B_A)^2-\Tr\B_A^2(X)},
  \end{align}
  where $\Rb$ is the curvature tensor of $M$.
\end{theorem}

\begin{proof}
  The Ricci curvature of $\Sigma$ is defined as
  \begin{align*}
    \R^p_b = g^{ac}g^{pd}g\paraa{R(e_c,e_d)e_b,e_a}
  \end{align*}
  and from Gauss' equation (\ref{eq:GaussEquation}) it follows that
  \begin{align*}
    \R^p_b = g^{pd}g^{ac}\gb\paraa{\Rb(e_c,e_d)e_b,e_a}
    + g^{ac}g^{pd}\sum_{A=1}^p\parab{h_{A,bd}h_{A,ac}-h_{A,bc}h_{A,ad}}.
  \end{align*}
  Since $(W_A)^a_b = g^{ac}h_{A,cb}$ one obtains
  \begin{align*}
    \R_b^p = g^{ac}g^{pd}\gb\paraa{\Rb(e_c,e_d)e_b,e_a}
    + \sum_{A=1}^p\bracketb{\paraa{\tr W_A}(W_A)^p_b-(W_A^2)_b^p},
  \end{align*}
  and as $\B_A(X)=\gamma^2 W_A(X)$ for any $X\in\TSigma$, and
  $\Tr\B_A=\gamma^2\tr W_A$, one has
  \begin{equation*}
    \R(X) = g^{ac}g^{pd}\gb\paraa{\Rb(e_c,e_d)e_b,e_a}X^be_p
    + \frac{1}{\gamma^4}\sum_{A=1}^p\bracketb{\paraa{\Tr \B_A}\B_A(X)-\B_A^2(X)}.
  \end{equation*}
  By expanding the first term as
  \begin{align*}
    g^{ac}&g^{pd}X^b\Rb_{ijkl}\paraa{\d_ax^i}\paraa{\d_bx^j}\paraa{\d_cx^k}\paraa{\d_dx^l}\paraa{\d_px^m}\d_m\\
    &=\frac{1}{g^2(n-1)!^2}\eps^{p\pv}\eps^{d\dv}g_{\pv\dv}\,\eps^{a\av}\eps^{c\cv}g_{\av\cv}
    X^b\Rb_{ijkl}\paraa{\d_ax^i}\paraa{\d_bx^j}\paraa{\d_cx^k}\paraa{\d_dx^l}\paraa{\d_px^m}\d_m\\
    &=\ldots=\frac{1}{\gamma^4}\paraa{\P^2}^{ik}\paraa{\P^2}^{lm}\Rb_{ijkl}X^j\d_m
  \end{align*}
  one obtains the desired result.
\end{proof}

\subsection{Construction of normal vectors}\label{sec:normalVectors}

\noindent The results in Section \ref{sec:nambuPoissonFormulation}
involve Nambu brackets of the embedding coordinates and the
components of the normal vectors. In this section we will prove that
one can replace sums over normal vectors by sums of Nambu
brackets of the embedding coordinates, thus providing expressions that
do not involve normal vectors.

It will be convenient to introduce yet another multi-index; namely, we
let $\alpha=i_1\ldots i_{p-1}$ consist of $p-1$ indices all taking
values between $1$ and $m$.

\begin{proposition}\label{prop:normalvectors}
  For any value of the multi-index $\alpha$, the vector
  \begin{align}\label{eq:Zdef}
    Z_{\alpha}=\frac{1}{\gamma\paraa{n!\sqrt{(p-1)!}}}\gb^{ij}\eps_{jk_1\cdots k_n\alpha}\{x^{k_1},\ldots,x^{k_n}\}\d_i,
  \end{align}
  where $\eps_{i_1\cdots i_m}$ is the Levi-Civita tensor of $M$, is
  normal to $\TSigma$, i.e. $\gb(Z_{\alpha},e_a)=0$ for
  $a=1,2,\ldots,n$. For hypersurfaces ($p=1$), equation (\ref{eq:Zdef})
  defines a unique normal vector of unit length.
\end{proposition}

\begin{proof}
  To prove that $Z_\alpha$ are normal vectors, one simply notes that
  \begin{align*}
    \gamma\paraa{n!\sqrt{(p-1)!}}\gb(Z_\alpha,e_a) &= 
    \frac{1}{\rho}\eps^{a_1\cdots a_n}\eps_{jk_1\cdots k_n\alpha}\paraa{\d_ax^j}\paraa{\d_{a_1}x^{k_1}}\cdots\paraa{\d_{a_n}x^{k_n}}=0,
  \end{align*}
  since the $n+1$ indices $a,a_1,\ldots,a_n$ can only take on $n$
  different values and since
  $(\d_ax^j)(\d_{a_1}x^{k_1})\cdots(\d_{a_n}x^{k_n})$ is contracted
  with $\eps_{jk_1\cdots k_n\alpha}$ which is completely antisymmetric
  in $j,k_1,\ldots,k_n$. Let us now calculate $|Z|^2\equiv\gb(Z,Z)$
  when $p=1$. Using that\footnote{In our convention, no combinatorial factor is included in the anti-symmetrization; for instance,  
    $\delta^{[i}_{[k}\delta^{j]}_{l]}=\delta^i_k\delta^j_l-\delta^i_l\delta^j_k$.}
  \begin{align*}
    \eps_{ik_1\cdots k_n}\eps^{il_1\cdots l_n} = \delta^{[l_1}_{[k_1}\cdots\delta^{l_n]}_{k_n]}
  \end{align*}
  one obtains
  \begin{align*}
    |Z|^2 &= \frac{1}{\gamma^2n!^2}\gb_{l_1l_1'}\cdots\gb_{l_nl_n'}
    \eps_{ik_1\cdots k_n}\eps^{il_1\cdots l_n}
    \{x^{k_1},\ldots,x^{k_n}\}\{x^{l_1'},\ldots,x^{l_n'}\}\\
    &=\frac{1}{\gamma^2n!^2}\gb_{l_1l_1'}\cdots\gb_{l_nl_n'}
    \delta^{[l_1}_{[k_1}\cdots\delta^{l_n]}_{k_n]}
    \{x^{k_1},\ldots,x^{k_n}\}\{x^{l_1'},\ldots,x^{l_n'}\}\\
    &=\frac{1}{\gamma^2n!}\{x^{l_1},\ldots,x^{l_n}\}
    \gb_{l_1l_1'}\cdots\gb_{l_nl_n'}\{x^{l_1'},\ldots,x^{l_n'}\}\\
    &=\frac{1}{\gamma^2n!}(n-1)!\Tr\P^2 = \frac{1}{\gamma^2n!}(n-1)!n\gamma^2=1,
  \end{align*}
  which proves that $Z$ has unit length.
\end{proof}

\noindent If the codimension is greater than one, $Z_\alpha$ defines
more than $p$ non-zero normal vectors that do not in general fulfill any
orthonormality conditions. In principle, one can now apply the
Gram-Schmidt orthonormalization procedure to obtain a set of $p$
orthonormal vectors. However, it turns out that one can use $Z_\alpha$
to construct another set of normal vectors, avoiding explicit use of
the Gram-Schmidt procedure; namely, introduce
\begin{align*}
  \Z_{\alpha}^{\beta} = \gb(Z_{\alpha},Z^\beta),
\end{align*}
and consider it as a matrix over multi-indices $\alpha$ and
$\beta$. As such, the matrix is symmetric (with respect to
$\gb_{\alpha\beta}\equiv \gb_{i_1j_1}\cdots\gb_{i_{p-1}j_{p-1}}$) and
we let ${E_\alpha}^\beta,\mu_\alpha$ denote orthonormal eigenvectors
(i.e. $\gb_{\delta\sigma}E_\alpha^\delta E_\beta^\sigma=\delta_{\alpha\beta}$) 
and their corresponding
eigenvalues. Using these eigenvectors to define
\begin{align*}
  \Nh_\alpha = E^{\beta}_\alpha Z_\beta
\end{align*}
one finds that
$\gb(\Nh_\alpha,\Nh_\beta)=\mu_\alpha\delta_{\alpha\beta}$, i.e. the
vectors are orthogonal.

\begin{proposition}\label{prop:Zprojection}
  For $\Z_\alpha^\beta=\gb_{ij}Z^i_\alpha Z^{j\beta}$ it holds that
  \begin{align}
    \Z_\alpha^\delta\Z_\delta^\beta = \Z_\alpha^\beta\label{eq:Zidempot}\\
    \Z_\alpha^\alpha = p.\label{eq:Ztrace}
  \end{align}
\end{proposition}

\begin{proof}
  Both statements can be easily proved once one has the following result
  \begin{align}
    Z^i_\alpha Z^{j\alpha} = \gb^{ij}-\frac{1}{\gamma^2}\paraa{\P^2}^{ij},\label{eq:ZZproj}
  \end{align}
  which is obtained by using that
  \begin{align*}
    \eps_{kk_1\cdots k_n\alpha}\eps^{ll_1\cdots l_n\alpha} = 
    (p-1)!\parab{\delta^{[l}_{[k}\delta^{l_1}_{k_1}\cdots\delta^{l_n]}_{k_n]}}.
  \end{align*}
  Formula
  (\ref{eq:Ztrace}) is now immediate, and to obtain
  (\ref{eq:Zidempot}) one notes that since $Z_\alpha\in\TSigma^\perp$
  it holds that $\P^2(Z_\alpha)=0$, due to the fact that $\P^2$ is
  proportional to the projection onto $\TSigma$.
\end{proof}

\noindent From Proposition \ref{prop:Zprojection} it follows that an
eigenvalue of $\Z$ is either 0 or 1, which implies that $\Nh_\alpha=0$
or $\gb(\Nh_\alpha,\Nh_\alpha)=1$, and that the number of non-zero
vectors is $\Tr\Z = \Z_\alpha^\alpha=p$. Hence, the $p$ non-zero
vectors among $\Nh_\alpha$ constitute an orthonormal basis of
$\TSigma^\perp$, and it follows that one can replace any sum over
normal vectors $N_A$ by a sum over the multi-index of $\Nh_\alpha$. 
As an example, let us work out some explicit expressions in the case
when $M=\reals^m$.

\begin{proposition}
  Assume that $M=\reals^m$ and that all repeated indices are summed
  over. For any $X\in\TSigma$ one has
  \begin{align}
    &\sum_{A=1}^p\paraa{\Tr\B_A}\B_A(X)^i = \frac{1}{(n-1)!^2}\Pi^{jk}
    \{\{x^j,\xv^J\},\xv^J\}\{x^i,\xv^I\}\{X^k,\xv^I\}\label{eq:trBABAx}\\
    &\sum_{A=1}^p\B_A^2(X)^i=\frac{1}{(n-1)!^2}\Pi^{jk}
    \{x^i,\xv^I\}\{\{x^j,\xv^J\}\{X^k,\xv^J\},\xv^I\}\\
    &\sum_{A=1}^p\paraa{\Tr\B_A}N_A^i=
    \frac{(-1)^n}{(n-1)!}\Pi^{ik}\{\{x^k,\xv^I\},\xv^I\}
  \end{align}
  where
  \begin{align}
    \Pi^{ij} = \delta^{ij}-\frac{1}{\gamma^2}\para{\P^2}^{ij}
  \end{align}
  is the projection onto the normal space.
\end{proposition}

\begin{proof}
  Let us prove formula (\ref{eq:trBABAx}); the other formulas can be
  proven analogously. One rewrites
  \begin{align*}
    \paraa{\Tr\B_A}\B_A(X)^i &=
    \frac{1}{(n-1)!^2}\{x^j,\xv^J\}\{\xv^J,n^j_A\}\{x^i,\xv^I\}\{\xv^I,n_A^k\}X^k\\
    &= \frac{1}{(n-1)!^2}n_A^jn_A^k\{\xv^J,\{x^j,\xv^J\}\}
    \{x^i,\xv^I\}\{\xv^I,X^k\}
  \end{align*}
  since $n_A^j\{x^j,\xv^J\}=n_A^kX^k=0$, due to the fact that $N_A$
  is a normal vector. By replacing $n_A^jn_A^k$ with
  $\Nh_\alpha^j\Nh_\alpha^k$ and using the fact that
  \begin{align*}
    \Nh_\alpha^i\Nh_{\alpha}^j = \delta^{ij}-\frac{1}{\gamma^2}\paraa{\P^2}^{ij}
  \end{align*}
  one obtains
  \begin{equation*}
    \paraa{\Tr\B_A}\B_A(X)^i = \frac{1}{(n-1)!^2}
    \Pi^{jk}\{\{x^j,\xv^J\},\xv^J\}\{x^i,\xv^I\}\{X^k,\xv^I\}.\qedhere
  \end{equation*}
\end{proof}

\noindent For hypersurfaces in $\reals^{n+1}$, the ``Theorema Egregium'' states
that the determinant of the Weingarten map, i.e the ``Gaussian
curvature'', is an invariant (up to a sign when $\Sigma$ is
odd-dimensional) under isometries (this is in fact also true for
hypersurfaces in a manifold of constant sectional curvature). From
Proposition \ref{prop:TrPBST} we know that one can express $\det W_A$
in terms of $\Tr\S_A\T_A$. 
\begin{proposition}\label{prop:HyperSurfaceGaussian}
  Let $\Sigma$ be a hypersurface in $\reals^{n+1}$ and let $W$ denote
  the Weingarten map with respect to the unit normal
  \begin{align*}
    Z = \frac{1}{\gamma n!}\gb^{ij}\eps_{jkK}\{x^k,\xv^K\}.
  \end{align*}
  Then one can write $\det W$ as 
  \begin{align*}
    \det W = -\frac{1}{\gamma(\gamma n!)^{n+1}}&\sum
    \eps_{ilL}\eps_{j_1k_1K_1}\cdots\eps_{j_{n-1}k_{n-1}K_{n-1}}\\
    &\times\{x^i,\{x^{k_1},\xv^{K_1}\},\ldots,\{x^{k_{n-1}},\xv^{K_{n-1}}\}\}
    \{\xv^J,\{x^l,\xv^L\}\}.
    % -\frac{1}{n\gamma^2}\Tr\S_A\T_A    = 
  \end{align*}
\end{proposition}

\noindent In fact, one can express all the elementary symmetric
functions of the principle curvatures in terms of Nambu brackets as
follows: The elementary symmetric functions of the eigenvalues
of $W$ is given (up to a sign) as the coefficients of the polynomial
$\det(W-t\mid)$. Since $\B(X)=0$ for all $X\in\TSigma^\perp$ and
$\B(X)=\gamma^2W(X)$ for all $X\in\TSigma$, it holds that
\begin{align*}
  -t\det(W-t\mid_n) = \det(\gamma^{-2}\B-t\mid_{n+1})
  =\frac{1}{\gamma^{2(n+1)}}\det(\B-t\gamma^2\mid_{n+1})
\end{align*}
which implies that the coefficient of $t^k$ in $\det(W-t\mid)$ is
given by the coefficient of $t^{k+1}$ in
$-\det(\B-t\gamma^2\mid)\gamma^{2(n-k)}$.

\subsection{The Codazzi-Mainardi equations}\label{sec:CodazziMainardi}

\noindent When studying the geometry of embedded manifolds, the Codazzi-Mainardi
equations are very useful. In this section we reformulate these equations
in terms of Nambu brackets.

The Codazzi-Mainardi equations express the normal component
of $\Rb(X,Y)Z$ in terms of the second fundamental forms; namely
\begin{align}\label{eq:CMh}
  \begin{split}
    \gb\paraa{&\Rb(X,Y)Z,N_A} = \paraa{\nabla_Xh_A}(Y,Z) - \paraa{\nabla_Yh_A}(X,Z)\\
    &+\sum_{A=1}^p\bracketb{\gb(D_XN_B,N_A)h_B(Y,Z)-\gb(D_YN_B,N_A)h_B(X,Z)},
  \end{split}
\end{align}
for $X,Y,Z\in\TSigma$ and $A=1,\ldots,p$. Defining 
\begin{align}
  \begin{split}
    \W_A&(X,Y) = \paraa{\nabla_XW_A}(Y)-\paraa{\nabla_Y W_A}(X)\\
    &+\sum_{B=1}^p\bracketb{\gb(D_XN_B,N_A)W_B(Y)-\gb(D_YN_B,N_A)W_B(X)}
  \end{split}
\end{align}
one can rewrite the Codazzi-Mainardi equations as follows.
\begin{proposition}\label{prop:CMWPi}
  Let $\Pi$ denote the projection onto $\TSigma^\perp$.  Then the
  Codazzi-Mainardi equations are equivalent to
  \begin{align}\label{eq:CMCA}
    \W_A(X,Y) = -(\mid-\Pi)\paraa{\Rb(X,Y)N_A}
  \end{align}
  for $X,Y\in\TSigma$ and $A=1,\ldots,p$.
\end{proposition}

\begin{proof}
  Since $h_A(X,Y)=\gb(W_A(X),Y)$ (by Weingarten's equation) one can
  rewrite (\ref{eq:CMh}) as
  \begin{align}
    \gb\paraa{\W_A(X,Y),Z} = \gb\paraa{\Rb(X,Y)Z,N_A},
  \end{align}
  and since $\gb(\Rb(X,Y)Z,N_A) = -\gb(\Rb(X,Y)N_A,Z)$ this becomes
  \begin{align}
    \gb\paraa{\W_A(X,Y)+\Rb(X,Y)N_A,Z} = 0.
  \end{align}
  That this holds for all $Z\in\TSigma$ is equivalent to saying that 
  \begin{align}
    (\mid-\Pi)\paraa{\W_A(X,Y)+\Rb(X,Y)N_A} = 0,
  \end{align}
  from which (\ref{eq:CMCA}) follows since $\W_A(X,Y)\in\TSigma$.
\end{proof}

\noindent Note that since $\gamma^{-2}\P^2$ is the projection onto
$\TSigma$ one can write (\ref{eq:CMCA}) as
\begin{align}\label{eq:CMPgamma}
  \gamma^2\W_A(X,Y) = -\P^2\paraa{\Rb(X,Y)N_A}.
\end{align}
\noindent Since both $W_A$ and $D_X$ can be expressed in terms of
$\B_A$, one obtains the following expression for $\W_A$:
\begin{proposition}
  For $X,Y\in\TSigma$ one has
  \begin{align*}
    \gamma^2\W_A(X,Y) = &\paraa{\nablab_X\B_A}(Y)-\paraa{\nablab_Y\B_A}(X)\\
    &-\frac{1}{\gamma^2}\bracketb{\paraa{\nabla_X\gamma^2}\B_A(Y)-\paraa{\nabla_Y\gamma^2}\B_A(X)}\\
    &+\frac{1}{\gamma^2}\sum_{B=1}^p\bracketb{\gb\paraa{\B_A(N_B),X}\B_B(Y)-\gb\paraa{\B_A(N_B),Y}\B_B(X)}.
  \end{align*}
\end{proposition}

\noindent As the aim is to express the Codazzi-Mainardi equations in
terms of Nambu brackets, we will introduce maps $\C_A$ that
is defined in terms of $\W_A$ and can be written as expressions
involving Nambu brackets.

\begin{definition}
  The maps $\C_A:C^\infty(\Sigma)\times\cdots\times
  C^\infty(\Sigma)\to \TSigma$ are defined as
  \begin{align}
    \C_A(f_1,\ldots,f_{n-2}) = \frac{1}{2\rho}\eps^{aba_1\cdots a_{n-2}}
    \W_A(e_a,e_b)\paraa{\d_{a_1}f_1}\cdots\paraa{\d_{a_{n-2}}f_{n-2}}
  \end{align}
  for $A=1,\ldots,p$ and $n\geq 3$. When $n=2$, $\C_A$ is defined as
  \begin{align*}
    \C_A = \frac{1}{2\rho}\eps^{ab}\W_A(e_a,e_b).
  \end{align*}
\end{definition}

\begin{proposition}\label{prop:CANPbracket}
  Let $\{g_1,g_2\}_f\equiv\{g_1,g_2,f_1,\ldots,f_{n-2}\}$. Then
  \begin{align*}
    \C_A(f_1,\ldots,&f_{n-2})^i = 
    \pb{\gamma^{-2}(\B_A)^i_k,x^k}_f
    +\frac{1}{\gamma^2}\pb{x^j,x^l}_f\bracketb{\Gammab^i_{jk}(\B_A)^k_l-(\B_A)^i_k\Gammab^k_{jl}}\\
    &-\frac{1}{\gamma^2}\sum_{B=1}^p\bracketb{
      \pb{n_A^k,x^l}_f(\B_B)^i_l+\Gammab^k_{lj}\pb{x^l,x^m}_fn_A^j(\B_B)^i_m
    }(n_B)_k.
  \end{align*}
\end{proposition}

\begin{remark}
  \noindent In case $\Sigma$ is a hypersurface, the expression for $\C\equiv \C_1$ simplifies to
  \begin{align*}
    \C(f_1,\ldots,f_{n-2})^i = 
    &\pb{\gamma^{-2}\B^i_k,x^k}_f
    +\frac{1}{\gamma^2}\pb{x^j,x^l}_f\bracketb{\Gammab^i_{jk}\B^k_l-\B^i_k\Gammab^k_{jl}},
  \end{align*}
  since $D_XN=0$.  
\end{remark}

\noindent It follows from Proposition \ref{prop:CMWPi} that we can
reformulate the Codazzi-Mainardi equations in terms of $\C_A$:

\begin{theorem}\label{thm:CMNambu}
  For all $f_1,\ldots,f_{n-2}\in C^\infty(\Sigma)$ it holds that
  \begin{align}\label{eq:CMNambu}
    \gamma^2\C_A(f_1,\ldots,f_{n-2}) = (\P^2)^{i}_j\bracketb{\{x^k,\Gammab^j_{kj'}\}_f-\pb{x^k,x^l}_f\Gammab^m_{lj'}\Gammab^j_{km}}n_A^{j'}\d_i,
  \end{align}
  for $A=1,\ldots,p$, where $\{g_1,g_2\}_f=\{g_1,g_2,f_1,\ldots,f_{n-2}\}$.
\end{theorem}

\begin{proof}
  As noted previously, one can write the Codazzi-Mainardi equations as
  \begin{align*}
    \gamma^2\W_A(X,Y) = -\P^2\paraa{\Rb(X,Y)N_A}.
  \end{align*}
  That the above equation holds for all $X,Y\in\TSigma$ is equivalent to saying that
  \begin{align*}
    \gamma^2\frac{1}{2\rho}\eps^{aba_1\cdots a_{n-2}}\W_A(e_a,e_b)
    =-\frac{1}{2\rho}\eps^{aba_1\cdots a_{n-2}}\P^2\paraa{\Rb(e_a,e_b)N_A}
  \end{align*}
  for all values of $a_1,\ldots,a_{n-2}\in\{1,\ldots,n\}$; furthermore, this is equivalent to
  \begin{align*}
    \gamma^2\C_A(f_1,\ldots,f_{n-2}) = 
    -\frac{1}{2\rho}\eps^{aba_1\cdots a_{n-2}}\P^2\paraa{\Rb(e_a,e_b)N_A}
    (\d_{a_1}f_1)\cdots(\d_{a_{n-2}}f_{n-2})
  \end{align*}
  for all $f_1,\ldots,f_{n-2}\in C^\infty(\Sigma)$. It is now straightforward to show that
\begin{align*}
    -\frac{1}{2\rho}\eps^{aba_1\cdots a_{n-1}}&\paraa{\Rb(e_a,e_b)N_A}^i(\d_{a_1}f_1)\cdots(\d_{a_{n-2}}f_{n-2})\\
    &=\parab{\{x^k,\Gammab^i_{kj}\}_f-\pb{x^k,x^l}_f\Gammab^m_{lj}\Gammab^i_{km}}n_A^{j},
  \end{align*}
  which proves the statement.
\end{proof}

\noindent If $M$ is a space of constant curvature (in which case $\gb(\Rb(X,Y)Z,N_A)=0$), then Theorem \ref{thm:CMNambu} states that
\begin{align}
  \C_A(f_1,\ldots,f_{n-2}) = 0
\end{align}
for all $f_1,\ldots,f_{n-2}\in C^\infty(\Sigma)$. Furthermore, if $M=\reals^m$, then (\ref{eq:CMNambu}) becomes
\begin{align}
  \gamma^2\pb{\gamma^{-2}(\B_A)^i_k,x^k}_f-\sum_{B=1}^p\bracketb{
    \pb{n_A^k,x^l}_f(\B_B)^i_l}(n_B)_k = 0.
\end{align}

\subsection{Covariant derivatives}\label{sec:covariantDerivatives}

Equation (\ref{eq:GaussformulaB}) tells us that knowing $\nablab_XY$,
for $X,Y\in\TSigma$, one can compute $\nabla_XY$ through the formula
\begin{align*}
  \nabla_XY = \nablab_XY - \frac{1}{\gamma^2}\sum_{A=1}^p\gb\paraa{\B_A(X),Y}N_A,
\end{align*}
which requires explicit knowledge about the normal vectors. Are there
other quantities involving $\nabla$ that can be computed solely in
terms of the embedding coordinates? We will now show that the two derivations
\begin{align}
  &D^I(u)\equiv\frac{1}{\gamma\sqrt{(n-1)!}}\{u,\xv^I\}\\
  &\D^i(u)\equiv\gb_{IJ}D^I(x^i)D^J(u),
\end{align}
can be considered as analogues of covariant derivatives on
$\Sigma$. Their indices are lowered by the ambient metric
$\gb_{ij}$. Let us start by showing that several standard formulas
involving covariant derivatives with contracted indices also hold for
our newly defined derivations.
\begin{proposition}\label{prop:covderivFormulas}
  For $u,v\in C^\infty(\Sigma)$ it holds that
  \begin{align}
    \nabla u &= \D^i(u)\d_i=D_I(u)D^I(x^i)\d_i\\
    g\paraa{\nabla u,\nabla v} &= \D_i(u)\D^i(v)=D_I(u)D^I(v)\\
    \Delta(u)&=\D_i\D^i(u)=D_ID^I(u)\\
    |\nabla^2u|^2&=\D_i\D^j(u)\D_j\D^i(u)=D_ID^J(u)D_JD^I(u)\label{eq:nablaSquSq}
    % =g^{ap}g^{bq}\paraa{\nabla_a\nabla_bu}\paraa{\nabla_p\nabla_qu}
  \end{align}
\end{proposition}

\begin{proof}
  The most convenient way of proving the above identities is to work
  in a coordinate system where $u^1,\ldots,u^n$ are normal
  coordinates. In particular, this implies that $\Gamma^a_{bc}=0$,
  which is equivalent to $\gb_{ij}(\d_ax^i)\d^2_{bc}x^j=0$. Let us now
  prove formula (\ref{eq:nablaSquSq}) for the operators $D^I$.

  Let us first note that in normal coordinate one obtains
  \begin{align*}
    |\nabla^2u|^2\equiv\paraa{\nabla_a\nabla_bu}\paraa{\nabla_c\nabla_d u}g^{ac}g^{bd}
    =g^{ac}g^{bd}\paraa{\d^2_{ab}u}\paraa{\d^2_{cd}u}.
  \end{align*}
  We now compute
  \begin{align*}
    &D_ID^J(u)D_JD^I(u) = 
    \frac{1}{\gamma^2(n-1)!^2}\{\gamma^{-1}\{u,\xv^J\},\xv^K\}\gb_{KI}
    \{\gamma^{-1}\{u,\xv^I\},\xv^L\}\gb_{LJ}\\
    &= \frac{1}{g^2(n-1)!^2}\eps^{a\av}\d_a\paraa{\eps^{p\pv}(\d_pu)(\d_{\pv}\xv^J)}\paraa{\d_{\av}\xv^K}\gb_{KI}
    \eps^{c\cv}\d_c\paraa{\eps^{q\qv}(\d_qu)(\d_{\qv}\xv^I)}\paraa{\d_{\cv}\xv^L}\gb_{LJ}
  \end{align*}
  The terms involving $\d_a\d_{\pv}\xv^J$ and $\d_c\d_{\qv}\xv^I$
  vanish since they appear in combinations such as
  $(\d_a\d_{\pv}\xv^J)(\d_{\cv}\xv^L)\gb_{LJ}$ which is zero due to
  the presence of a normal coordinate system. Thus,
  \begin{align*}
    D_ID^J(u)D_JD^I(u) &= 
    \frac{1}{g^2(n-1)!^2}\eps^{a\av}\eps^{q\qv}g_{\av\qv}\eps^{p\pv}\eps^{c\cv}g_{\pv\cv}
    \paraa{\d^2_{ap}u}\paraa{\d^2_{cq}u}\\
    &=g^{aq}g^{pc}\paraa{\d^2_{ap}u}\paraa{\d^2_{cq}u}=|\nabla^2u|^2.
  \end{align*}
  The other formulas can be proved analogously.
\end{proof}

\noindent By definition, the curvature tenor of $\Sigma$ arises when
one commutes two covariant derivatives. In light of Theorem
\ref{thm:ricciCurvature}, one may ask if there is a similar Nambu
bracket relation which gives rise to the Ricci curvature. A particular
example that introduces curvature is the following
\begin{equation}\label{eq:curvatureEq}
  (\nabla^au)\nabla_a\nabla_b\nabla^bu=(\nabla^au)\nabla_b\nabla_a\nabla^bu
  -g(\R(\nabla u),\nabla u).
\end{equation}

\noindent Since $(\nabla^au)\nabla_a\nabla_b\nabla^bu=g(\nabla
u,\nabla\Delta u)$, it follows from Proposition
\ref{prop:covderivFormulas} that one can write it as
\begin{equation}
  (\nabla^au)\nabla_a\nabla_b\nabla^bu=
  \D_i(u)\D^i\D_j\D^j(u) = D_I(u)D^ID_JD^J(u),
\end{equation}
and the term in (\ref{eq:curvatureEq}) involving the Ricci curvature
is written in terms of Nambu brackets through Theorem
\ref{thm:ricciCurvature}. Using the relation
\begin{equation}
  \Delta\paraa{|\nabla u|^2} = 2\paraa{\nabla^au}\nabla^b\nabla_a\nabla_b u
  +2|\nabla^2u|^2,
\end{equation}
and (\ref{eq:nablaSquSq}) one obtains
\begin{align*}
  \paraa{\nabla^au}\nabla^b\nabla_a\nabla_b u &= \frac{1}{2}\D_i\D^i\paraa{\D_j(u)\D^j(u)}
  -\D_i\D^j(u)\D_j\D^i(u)\\
  &=\D_i(u)\D^j\D_j\D^i(u)+\fatcom{\D_i,\D^j}(u)\D_i\D^j(u),
\end{align*}
where $\fatcom{\D^i,\D^j}$ denotes the commutator with respect to composition of operators.
Thus, we arrive at the following result:
\begin{proposition}\label{prop:covariantDRicci}
  Let $\R$ be the Ricci curvature of $\Sigma$ and let $u\in
  C^\infty(\Sigma)$. Then it holds that
  \begin{align*}
    &\D_i(u)\D^i\D_j\D^j(u) = \D_i(u)\D^j\D_j\D^i(u)+\fatcom{\D_i,\D^j}(u)\D_i\D^j(u)
    -g(\R(\nabla u),\nabla u)\\
    &D_I(u)D^ID_JD^J(u) = D_I(u)D^JD_JD^I(u)+\fatcom{D_I,D^J}(u)D_ID^J(u)
    -g(\R(\nabla u),\nabla u).  
  \end{align*}
\end{proposition}

\noindent Note that it follows from Theorem \ref{thm:ricciCurvature}
that the term $g(\R(\nabla u),\nabla u)$ can be written in
terms of Nambu brackets. If the formulas in Proposition
\ref{prop:covariantDRicci} are integrated, one arrives at expressions
whose index structure closely resembles that of equation
(\ref{eq:curvatureEq}). Namely, by partial integration one obtains
\begin{align*}
  \int\parab{D_I(u)D^JD_JD^I(u)+\fatcom{D_I,D^J}(u)D_ID^J(u)}\sqrt{g}
  =\int D_I(u)D_JD^ID^J(u)\sqrt{g},
\end{align*}
which implies
\begin{align}
  \int D^I(u)D_ID^JD_J(u)\sqrt{g} = 
  \int\parab{D_I(u)D_JD^ID^J(u)-g(\R(\nabla u),\nabla u)}\sqrt{g}.
\end{align}
Note that since the operators $D^I$ contain a factor of $\gamma^{-1}$,
the integration is actually performed with respect to $\rho$, as
$\gamma^{-1}\sqrt{g}=\rho$. 

The derivations $D^I$ and $\D^i$ have indices of the ambient space
$M$; do they exhibit any tensorial properties?  The object $\D^i(u)$
transforms as a tensor in the ambient space $M$, i.e.
\begin{align*}
  \D_y^i(u) &= \frac{1}{\gamma^2(n-1)!}\{u,\yv^I\}\gb_{IJ}(y)\{y^i,\yv^J\}\\
  &=\frac{1}{\gamma^2(n-1)!}\frac{\d y^i}{\d x^k}\{u,\xv^I\}\gb_{IJ}(x)\{x^k,\xv^J\}
  =\frac{\d y^i}{\d x^k}\D_x^k(u),
\end{align*}
but this does not hold for the next order derivative
$\D^i\D^j(u)$ due to the second derivatives on the embedding
functions. One can however ``covariantize'' this object by adding
extra terms. 
\begin{proposition}
  Define $\nabla^{ij}$ acting on $u\in C^\infty(\Sigma)$ as
  \begin{align}
    \nabla^{ij}(u) = \frac{1}{2}\parab{\D^i\D^j(u)+\D^j\D^i(u)
    -\D^u\paraa{\D^i(x^j)}},
  \end{align}
  where
  $\D^u(f)=\frac{1}{\gamma^2(n-1)!}\{f,\xv^I\}\gb_{IJ}\{u,\xv^J\}$. Then
  $\nabla^{ij}(u)$ transforms as a tensor in $M$, i.e.
  \begin{align*}
    \nabla_y^{ij}(u) = \frac{\d y^i}{\d x^k}\frac{\d y^j}{\d x^l}\nabla^{kl}_x(u),
  \end{align*}
  and for all $X,Y\in\TSigma$ it holds that
  \begin{align*}
    \nabla_{ij}(u)X^iY^j = \paraa{\nabla_a\nabla_bu}X^aY^b.
  \end{align*}
  In particular, this implies that $\gb_{ij}\nabla^{ij}(u)=\Delta(u)$ and
  $\gb_{ij}\gb_{kl}\nabla^{ik}(u)\nabla^{jl}(u)=|\nabla^2u|^2$.
\end{proposition}
% \noindent One can also check, for any two vectors $X,Y\in\TSigma$, that
% \begin{align*}
%   \nabla_{ij}(u)X^iY^j = \paraa{\nabla_a\nabla_bu}X^aY^b,
% \end{align*}
% which implies that the tensor $\nabla^{ij}(u)$ behaves in almost every
% respect as the ambient space equivalent of $\nabla^a\nabla^bu$.

\subsection{Embedded surfaces}\label{sec:surfaces}

\noindent Let us now turn to the special case when $\Sigma$ is a
surface. For surfaces, the tensors $\P$, $\S_A$ and $\T_A$ are themselves
maps from $TM$ to $TM$, and $\S_A$ coincides with
$\T_A$. Moreover, since the second fundamental forms can be considered
as $2\times 2$ matrices, one has the identity
\begin{align*}
  2\det W_A = \paraa{\tr W_A}^2-\tr W_A^2,
\end{align*}
which implies that the scalar curvature can be written as
\begin{align*}
  R &= \frac{1}{\gamma^4}\paraa{\P^2}^{ik}\paraa{\P^2}^{jl}\Rb_{ijkl} 
  %g^{ac}g^{bd}\gb\paraa{\Rb(e_c,e_d)e_b,e_a}
  + 2\sum_{A=1}^p\det W_A.
  %&= 2\frac{\gb\paraa{\Rb(e_1,e_2)e_2,e_1}}{g}+ 2\sum_{A=1}^p\det W_A.
\end{align*}
Thus, defining the Gaussian curvature $K$ to be one half of the above
expression (which also coincides with the sectional curvature), one obtains
\begin{align}
  K = \frac{1}{2\gamma^4}\paraa{\P^2}^{ik}\paraa{\P^2}^{jl}\Rb_{ijkl}-\frac{1}{2\gamma^2}\sum_{A=1}^p\Tr\S_A^2,
\end{align}
which in the case when $M=\reals^m$ becomes
\begin{align}
  K =-\frac{1}{2\gamma^2}\sum_{A=1}^p\sum_{i,j=1}^m\{x^i,n_A^j\}\{x^j,n_A^i\},
\end{align}
and by using the normal vectors $Z_\alpha$ the expression for $K$ can
be written as
\begin{equation}
  \begin{split}
    K &= -\frac{1}{8\gamma^4(p-1)!}
    \sum\eps_{jklI}\eps_{imnI}\{x^i,\{x^k,x^l\}\}\{x^j,\{x^m,x^n\}\}\\
    &=\frac{1}{\gamma^4}\parac{\frac{1}{2}\{\{x^j,x^k\},x^k\}\{\{x^j,x^l\},x^l\}
      -\frac{1}{4}\{\{x^j,x^k\},x^l\}\{\{x^j,x^k\},x^l\}}.    
  \end{split}
\end{equation}

\noindent To every Riemannian metric on $\Sigma$ one can associate an
almost complex structure $\J$ through the formula
\begin{equation*}
  \J(X) = \frac{1}{\sqrt{g}}\eps^{ac}g_{cb}X^be_a,
\end{equation*}
and since on a two dimensional manifold any almost complex structure
is integrable, $\J$ is a complex structure on $\Sigma$. For $X\in TM$ one has
\begin{align}
  \P(X) = -\frac{1}{\gamma\sqrt{g}}\gb\paraa{X,e_a}\eps^{ab}e_b,
\end{align}
and it follows that one can express the complex structure in terms of $\P$.

\begin{theorem}\label{thm:complexstructure}
  Defining $\JM(X)=\gamma\P(X)$ for all $X\in TM$ it holds that
  $\J_M(Y)=\J(Y)$ for all $Y\in\TSigma$. That is, $\gamma\P$ defines a
  complex structure on $\TSigma$.
\end{theorem}

\noindent Let us now turn to the Codazzi-Mainardi equations for
surfaces. In this case, the map $\C_A$ becomes a tangent vector and 
one can easily see in Proposition \ref{prop:CANPbracket} that the sum
in the expression for $\C_A$ can be written in a slightly more compact form, namely
\begin{align*}
  \C_A = &\pb{\gamma^{-2}(\B_A)^i_k,x^k}\d_i
  +\frac{1}{\gamma^2}\pb{x^j,x^l}\bracketb{\Gammab^i_{jk}(\B_A)^k_l-(\B_A)^i_k\Gammab^k_{jl}}\\
  &\qquad+\frac{1}{\gamma^2}\sum_{B=1}^p\B_B\S_A(N_B).
\end{align*}
\noindent Thus, for surfaces embedded in $\reals^m$ the Codazzi-Mainardi equations become
\begin{align*}
  \sum_{j,k=1}^m\pb{\gamma^{-2}\{x^i,x^j\}\{x^j,n_A^k\},x^k}\d_i+\frac{1}{\gamma^2}\sum_{B=1}^p\B_B\S_A(N_B)=0,
\end{align*}
and in $\reals^3$ one has
\begin{align}
  \sum_{j,k=1}^3\pba{\gamma^{-2}\{x^i,x^j\}\{x^j,n^k\},x^k} = 0.
\end{align}

\noindent Let us note that one can rewrite these equations using the following result:

\begin{proposition}
  For $M=\reals^m$ and $i=1,\ldots,m$ it holds that
  \begin{align}
    \sum_{j,k=1}^m\pba{f\{x^i,x^j\}\{x^j,n^k\},x^k} = 
    \sum_{j,k=1}^m\pba{f\{x^i,x^j\}\{x^j,x^k\},n^k}
  \end{align}
  for any normal vector $N=n^i\d_i$ and any $f\in C^\infty(\Sigma)$.
\end{proposition}

\begin{proof}
  We start by recalling that for any $g\in C^\infty(\Sigma)$ it holds
  that $\sum_{i=1}^m\{g,x^i\}n^i=0$, since it involves the scalar product
  $\gb(e_a,N)$. Moreover, one also has
  \begin{align*}
    \sum_{k=1}^m\{x^k,n^k\} &= \sum_{k=1}^m\frac{1}{\rho}\eps^{ab}(\d_ax^k)(\d_bn^k)
    =\sum_{k=1}^m\frac{1}{\rho}\eps^{ab}\parab{\d_b\paraa{n^k\d_ax^k}-n^k\d^2_{ab}x^k}\\
    &=-\sum_{k=1}^m\frac{1}{\rho}\eps^{ab}n^k\d^2_{ab}x^k=0,
  \end{align*}
  which implies that $\sum_{k=1}^m\{x^k,gn^k\}=0$ for all $g\in C^\infty(\Sigma)$.
  By using the above identities together with the Jacobi identity, one
  obtains
  \begin{align*}
    \pba{f\{x^i,x^j\}\{x^j,n^k\},x^k} &= 
    f\{x^i,x^j\}\pba{\{x^j,n^k\},x^k}+\{x^j,n^k\}\pba{f\{x^i,x^j\},x^k}\\
    &= -f\{x^i,x^j\}\pba{\{x^k,x^j\},n^k}-n^k\pba{x^j,\{f\{x^i,x^j\},x^k\}}\\
    &= -f\{x^i,x^j\}\pba{\{x^k,x^j\},n^k} + n^k\pba{f\{x^i,x^j\},\{x^k,x^j\}}\\
    &= -f\{x^i,x^j\}\pba{\{x^k,x^j\},n^k} - \{x^k,x^j\}\pba{f\{x^i,x^j\},n^k}\\
    &= \pba{f\{x^i,x^j\}\{x^j,x^k\},n^k}.\qedhere
  \end{align*}
\end{proof}

\noindent Hence, one can rewrite the Codazzi-Mainardi equations for a surface in $\reals^3$ as
\begin{align}\label{eq:CMR3P}
  \sum_{j,k=1}^3\pba{\gamma^{-2}(\P^2)^{ik},n^k} = 0,
\end{align}
and it is straight-forward to show that 
\begin{align*}
  \sum_{i,j,k=1}^3\paraa{\d_cx^i}\pba{\gamma^{-2}(\P^2)^{ik},n^k} = \frac{1}{\rho}\eps^{ab}\nabla_ah_{bc},
\end{align*}
thus reproducing the classical form of the Codazzi-Mainardi equations.

Is it possible to verify (\ref{eq:CMR3P}) directly using only Poisson
algebraic manipulations? It turns out that that the Codazzi-Mainardi
equations in $\reals^3$ is an identity for arbitrary Poisson algebras,
if one assumes that a normal vector is given by
$\frac{1}{2\gamma}\eps_{ijk}\{x^j,x^k\}\d_i$.
\begin{proposition}
  Let $\{\cdot,\cdot\}$ be an arbitrary Poisson structure on
  $C^\infty(\Sigma)$. Given $x^1,x^2,x^3\in C^\infty(\Sigma)$
  it holds that
  \begin{align*}
    \sum_{j,k,l,n=1}^3\frac{1}{2}\eps_{kln}\pba{\gamma^{-2}\{x^i,x^j\}\{x^j,x^k\},\gamma^{-1}\{x^l,x^n\}}=0
  \end{align*}
  for $i=1,2,3$, where
  \begin{align*}
    \gamma^2 = \{x^1,x^2\}^2+\{x^2,x^3\}^2+\{x^3,x^1\}^2.
  \end{align*}
\end{proposition}

\begin{proof}
  Let $u,v,w$ be a cyclic permutation of $1,2,3$. In the following we
  do not sum over repeated indices $u,v,w$. Denoting by $\text{CM}^i$
  the $i$'th component of the Codazzi-Mainardi equation, one has
  \begin{align*}
    &\text{CM}^u = -\pba{\gamma^{-2}\paraa{\{x^u,x^v\}^2+\{x^w,x^u\}^2},\gamma^{-1}\{x^v,x^w\}}\\
    &\quad+\pba{\gamma^{-2}\{x^u,x^v\}\{x^v,x^w\},\gamma^{-1}\{x^u,x^v\}}
    +\pba{\gamma^{-2}\{x^u,x^w\}\{x^w,x^v\},\gamma^{-1}\{x^w,x^u\}}\\
    &\quad= -\pba{1-\gamma^{-2}\{x^v,x^w\}^2,\gamma^{-1}\{x^v,x^w\}}
    +\gamma^{-1}\{x^u,x^v\}\pba{\gamma^{-1}\{x^v,x^w\},\gamma^{-1}\{x^u,x^v\}}\\
    &\quad+\gamma^{-1}\{x^u,x^w\}\pba{\gamma^{-1}\{x^w,x^v\},\gamma^{-1}\{x^w,x^u\}}\\
    &\quad= \frac{1}{2}\pba{\gamma^{-1}\{x^v,x^w\},\gamma^{-2}\paraa{\gamma^2-\{x^v,x^w\}^2}}=0.\qedhere
  \end{align*}
\end{proof}

\noindent Let us end by noting that these results generalize to
arbitrary hypersurfaces in $\reals^{n+1}$. Namely, 
\begin{align*}
  &\{\gamma^{-2}\pba{x^i,\xv^J\}\{\xv^J,n^k\},x^k}_f
  =\{\gamma^{-2}\pba{x^i,\xv^J\}\{\xv^J,x^k\},n^k}_f,\\
  &(\d_cx^i)\pba{\gamma^{-2}\paraa{\P^2}^{ik},n^k}_f=
  -\frac{1}{\rho}\eps^{aba_1\cdots a_{n-2}}\paraa{\nabla_ah_{bc}}\paraa{\d_{a_1}f_1}\cdots\paraa{\d_{a_{n-2}}f_{n-2}},
\end{align*}
and
\begin{align*}
  \eps_{klL}\pba{\gamma^{-2}\{x^i,\xv^J\}\{\xv^J,x^k\},\gamma^{-1}\{x^l,\xv^L\}}_f=0
\end{align*}
for arbitrary $x^1,\ldots,x^{n+1}\in C^\infty(\Sigma)$.

\section{Matrix regularizations}\label{sec:matrixRegularizations}

\noindent In physics, ``fuzzy spaces'' have been used for a long time
to regularize quantum theories and to model non-commutativity,
originating in the study of a quantum theory of surfaces (membranes)
sweeping out 3-manifolds of vanishing mean curvature). The main idea
was to replace smooth functions on a surface by sequences of matrices,
approximating the Poisson algebra of functions with increasing
accuracy as the matrix dimension grows.  Since the expressions for
geometric quantities derived in Section
\ref{sec:nambuPoissonFormulation} uses only the Poisson algebraic
structure of the function algebra, it is natural to study their matrix
analogues in this context.

Let us start by introducing some notation. Let $N_1,N_2,\ldots$ be a
strictly increasing sequence of positive integers and let $\Ta$, for
$\alpha=1,2,\ldots$, be linear maps from $C^\infty(\Sigma)$ to
hermitian $\Na\times\Na$ matrices. Moreover, let
$\hbar:\reals\to\reals$ be a strictly positive decreasing function
such that $\limNinf N\hbar(N)$ converges, and set $\hbara=\hbar(\Na)$.
Introduce the operators
\begin{align*}
  \partial^f(h) = \{f,h\}
\end{align*}
as well as the matrix operators
\begin{align*}
  \dha^f(X) = \frac{1}{i\hbara}[X,\Ta(f)], 
\end{align*}
and write
\begin{align*}
  &\partial^{f_1\cdots f_k}(h) = \partial^{f_1}\partial^{f_2}\cdots\partial^{f_k}(h)\\
  &\dha^{f_1\cdots f_k}(X) = \dha^{f_1}\dha^{f_2}\cdots\dha^{f_k}(X).
\end{align*}
Let us now define what is meant by a matrix regularization of compact
surface.

\begin{definition}
  Let $N_1,N_2,\ldots$ be a strictly increasing sequence of positive
  integers, let $\{\Ta\}$ for $\alpha=1,2,\ldots$ be linear maps from
  $C^\infty(\Sigma,\reals)$ to hermitian $\Na\times \Na$ matrices and
  let $\hbar(N)$ be a real-valued strictly positive decreasing
  function such that $\limNinf N\hbar(N)<\infty$. Furthermore, let
  $\omega$ be a symplectic form on $\Sigma$ and let $\{\cdot,\cdot\}$
  denote the Poisson bracket induced by $\omega$.  

  If for all integers $1\leq l\leq k$, $\{\Ta\}$ has the following
  properties for all $f,f_1,\ldots,f_k,h\in C^\infty(\Sigma)$
  \begin{align}
    &\limainfty\norm{\Ta(f)}<\infty\label{eq:matrixNorm},\\
    &\limainfty\norm{\Ta(fh)-\Ta(f)\Ta(h)}=0,\label{eq:matrixProduct}\\
    &\limainfty\norm{\dha^{f_1\cdots f_l}\paraa{\Ta(f)}-\Ta\paraa{\partial^{f_1\cdots f_l}(f)}}=0\label{eq:matrixCommutator}\\
    &\limainfty 2\pi\hbara\Tr\Ta(f)=\int_\Sigma f \omega,\label{eq:matrixTrace}
  \end{align}
  where $||\cdot||$ denotes the operator norm and
  $\hbar_\alpha=\hbar(\Na)$, then we call the pair $(\Ta,\hbar)$ a
  \emph{$C^k$-convergent matrix regularization of
    $(\Sigma,\omega)$}. If $(\Ta,\ha)$ is $C^k$-convergent for all
  $k\geq 0$ then $(\Ta,\ha)$ is called a \emph{smooth matrix regularization of
    $(\Sigma,\omega)$}.
\end{definition}

\noindent In the following, when we speak of a matrix regularization
without any reference to the degree of convergence, we shall always
mean a $C^1$-convergent matrix regularization.

\begin{remark}
  In some cases, a $C^1$-convergent matrix regularization is
  automatically a smooth matrix regularization. For instance, if it
  holds that for any $f,h\in C^\infty(\Sigma)$ there exists $A_k(f,h)\in
  C^\infty(\Sigma)$ such that
  \begin{align*}
    \frac{1}{i\ha}[\Ta(f),\Ta(h)]=\sum_k c_{k,\alpha}(f,h)\Ta\paraa{A_k(f,h)},
  \end{align*}
  for some $c_{k,\alpha}(f,h)\in\reals$, then $C^k$-convergence implies
  $C^{k+1}$-convergence. The matrix regularizations for the sphere and
  the torus in Section \ref{sec:simpleExamples} both fall into this
  category. Hence, they are examples of smooth matrix
  regularizations. Note that one can easily destroy the smoothness of
  a matrix regularization by slightly deforming it, see Example
  \ref{ex:deformedFuzzyTorus}.
\end{remark}

\begin{definition}
  A sequence $\{\fha\}$ of $N_\alpha\times N_\alpha$ matrices
  \emph{converges to $f$} (or \emph{$C^0$-converges to $f$}) if 
  \begin{align}
    \limainfty\norm{\fha-\Ta(f)}=0.
  \end{align}
  Moreover, for any integer $k\geq 1$, a sequence $\{\fha\}$ of
  $N_\alpha\times N_\alpha$ matrices \emph{$C^k$-converges to $f$} if in addition
  \begin{align*}
    \limainfty\norm{\dha^{f_1\cdots f_l}(\fha)-\Ta\paraa{\partial
^{f_1\cdots f_l}(f)}}=0,
  \end{align*}
  for all $1\leq l\leq k$ and $f_1,\ldots,f_l\in C^\infty(\Sigma)$.
  If $\{\fha\}$ is $C^k$-convergent for all positive $k$ then we say
  that $\{\fha\}$ is a smooth sequence. 
\end{definition}

\begin{remark}
  If the matrix regularization is $C^k$-convergent, it is clear that
  the matrix sequence $\Ta(f)$ is $C^k$-convergent. It is however easy
  to construct, even in a smooth matrix regularization, $C^0$-convergent
  sequences that are not $C^1$-convergent; see Example
  \ref{ex:nonSmoothSequence}.
\end{remark}

\begin{definition}
  A $C^k$-convergent matrix regularization $(\Ta,\hbar)$ is called
  \emph{unital} if the sequence $\{\mid_{\Na}\}$ $C^k$-converges to the constant function $1$.
\end{definition}

\begin{remark}
  Although unital matrix regularizations seem natural, and all our
  examples fall into this category, it is easy to construct examples
  of non-unital matrix regularizations. Namely, let $(\Ta,\hbar)$ be a
  matrix regularization and consider the map $\tilde{T}^\alpha$
  defined by
  \begin{align*}
    \tilde{T}^\alpha(f) = 
    \begin{pmatrix}
      &    &          &  & 0\\
      &    & \Ta(f)   &  & \vdots \\
      &    &          &  & \\
    0 &         & \cdots & & 0 
    \end{pmatrix}.
  \end{align*}
  Then $(\tilde{T}^\alpha,\hbar)$ is a matrix regularization which is not unital, since
  \begin{align*}
    \limainfty\norm{\tilde{T}^\alpha(1)-\mid_{N_\alpha+1}}\geq 1.
  \end{align*}
\end{remark}

\begin{proposition}
  Let $(\Ta,\hbar)$ be a unital matrix regularization. Then
  \begin{align}
    \limainfty 2\pi N_\alpha\hbara = \int_\Sigma\omega.
  \end{align}
\end{proposition}

\begin{proof}
  Let us use formula (\ref{eq:matrixTrace}) with $f=1$.
  \begin{align*}
    \int_\Sigma\omega &= \limainfty 2\pi\hbara\Tr\Ta(1)
    =\limainfty 2\pi\hbara\Tr\bracketb{\Ta(1)+\mid_{N_\alpha}-\mid_{N_\alpha}}\\
    &= \limainfty\parab{2\pi\hbara N_\alpha + 2\pi\hbara\Tr(\Ta(1)-\mid_{N_\alpha})}
    = \limainfty 2\pi\hbara N_\alpha
  \end{align*}
  since 
  \begin{align*}
    \limainfty \abs{2\pi\hbara\Tr(\Ta(1)-\mid_{N_\alpha})}
    \leq \limainfty 2\pi\hbara N_\alpha\norm{\Ta(1)-\mid_{N_\alpha}} = 0,
  \end{align*}
  due to the fact that the matrix regularization is unital.
\end{proof}

\begin{proposition}\label{prop:arbitrarySequences}
  Let $(\Ta,\ha)$ be a $C^k$-convergent matrix regularization and
  assume that $\fha$ and $\hha$ $C^k$-converge to $f,h\in
  C^\infty(\Sigma)$ respectively. Then it holds that $a\fha+b\hha$
  $C^k$-converges to $af+bh$, for any $a,b\in\reals$, and $\fha\hha$
  $C^k$-converges to $fh$. 
  Furthermore, it holds that
  \begin{align}
    &\limainfty\norm{\fha} = \limainfty\norm{\Ta(f)}\label{eq:multiMatrixNorm}\\
    &\limainfty 2\pi\hbara\Tr\paraa{\fha\hha} = \int_{\Sigma}fh\omega\label{eq:multiMatrixTrace}.
  \end{align}
\end{proposition}

\begin{proof}
  The fact that $a\fh+b\hh$ $C^k$-converges to $af+bh$ follows
  directly from linearity of the maps $\Ta$. To prove (\ref{eq:multiMatrixNorm}) one
  uses the reverse triangle inequality to deduce
  \begin{align*}
    \limainfty\left| ||\fha||-\norm{\Ta(f)}\right|
    \leq \limainfty\norm{\fha-\Ta(f)}=0,
  \end{align*}
  since $\fha$ is assumed to converge to $f$. Let us continue
  by proving that $\fha\hha$ $C^0$-converges to $fh$, i.e.
  \begin{align*}
    &\limainfty\norm{\fha\hha-\Ta(fh)} = 
    \limainfty\norm{\fha\hha-\fha\Ta(h)+\fha\Ta(h)-\Ta(fh)}\\
    &\leq\limainfty\Big(\norm{\fha}\norm{\hha-\Ta(h)}
    +\norm{\fha\Ta(h)-\Ta(f)\Ta(h)+\Ta(f)\Ta(h)-\Ta(fh)}\Big)\\
    &\leq\limainfty\Big(
    \norm{\fha}\norm{\hha-\Ta(h)}
    +\norm{\fha-\Ta(f)}\norm{\Ta(h)}
    +\norm{\Ta(f)\Ta(h)-\Ta(fh)}
    \Big)\\
    &=0,
  \end{align*}
  since both $\{\fha\}$ and $\{\hha\}$ are $C^0$-convergent sequences
  and $||\fha||$ is bounded by (\ref{eq:multiMatrixNorm}). Using the
  face that $\fha\hha$ $C^0$-converges to $fg$, it is easy to prove
  (\ref{eq:multiMatrixTrace}) by computing
  \begin{align*}
    \limainfty 2\pi\hbara\Tr\fha\hha &= 
    \limainfty 2\pi\hbara\Tr\paraa{\fha\hha-\Ta(fh)+\Ta(fh)}\\
    &=\limainfty 2\pi\hbara\Tr\Ta(fh)) = \int_\Sigma fh\omega.
  \end{align*}
  Finally, we proceed by induction to show that $\fha\hha$
  $C^k$-converges to $fh$. Thus, assume that, for some $0\leq l<k$,
  $\uha\vha$ $C^l$-converges to $uv$ whenever $\uha$ and $\vha$
  $C^l$-converges to $u$ and $v$ respectively. Since
  \begin{align*}
    \dha^{f_1}(\fha\hha) = \paraa{\dha^{f_1}\fha}\hha+\fha\dha^{f_1}\hha
  \end{align*}
  we can use the induction hypothesis (together with the assumption
  that $\fha,\hha$ $C^{k>l}$-converges) to conclude that
  $\dha^{f_1}(\fha\hha)$ $C^l$-converges, which implies that
  $\fha\hha$ $C^{l+1}$-converges. Hence, it follows that
  $\fha\hha$ $C^k$-converges to $fh$.
\end{proof}

\noindent The above result allows one to easily construct sequences of
matrices converging to any sum of products of functions and Poisson
brackets. Namely, simply substitute for every factor in every term of
the sum, a sequence converging to that function, where Poisson
brackets of functions may be replaced by commutators of
matrices. Proposition \ref{prop:arbitrarySequences} then guarantees
that the matrix sequence obtained in this way converges to the sum of
the products of the corresponding functions, as long as the
appropriate level of convergence is assumed.

\begin{proposition}
  Let $(\Ta,\hbar)$ be a matrix regularization and let $\{\fha\}$ be a
  sequence converging to $f$. Then $\limainfty||\fha||=0$ if and only if $f=0$. 
\end{proposition}

\begin{proof}
  From Proposition \ref{prop:arbitrarySequences} it follows directly
  that if $\fha$ converges to $0$ then
  \begin{align*}
    \limainfty||\fha||=\limainfty||\Ta(0)||=0.
  \end{align*}
  Now, assume that $\limainfty||\fha||=0$. Then it holds that
  \begin{align*}
    \int f^2\omega = \limainfty 2\pi\ha\Tr\fha^2
    \leq \limainfty 2\pi\ha\Na||\fha^2||
    \leq \limainfty 2\pi\ha\Na||\fha||^2=0,
  \end{align*}
  from which we conclude that $f=0$.
\end{proof}

\begin{proposition}\label{prop:uniqueNormZero}
  Let $(\Ta,\hbar)$ be a matrix regularization and assume that
  $\{\fha\}$ $C^k$-converges to $f$. Then $\{\fha^\dagger\}$
  $C^k$-converges to $f$.
\end{proposition}

\begin{proof}
  Due to the fact that $||A||=||A^\dagger||$ one sees that
  \begin{align*}
    \limainfty&\norm{\dha^{f_1\cdots f_k}(\fha^\dagger)-\Ta\paraa{\partial^{f_1\cdots f_k}(f)}}
    =\limainfty\norm{\dha^{f_1\cdots f_k}(\fha^\dagger)^\dagger-\Ta\paraa{\partial^{f_1\cdots f_k}(f)}}\\
    &=\limainfty\norm{\dha^{f_1\cdots f_k}(\fha)-\Ta\paraa{\partial^{f_1\cdots f_k}(f)}}=0,
  \end{align*}
  since $\{\fha\}$ $C^k$-converges to $f$.
\end{proof}

\begin{proposition}
  Let $(\Ta,\hbar)$ be a unital matrix regularization and assume that
  $f$ is a nowhere vanishing function and that $\{\fha\}$
  $C^k$-converges to $f$. If $\fha^{-1}$ exists and
  $||\fha^{-1}||$ is uniformly bounded for all $\alpha$, then
  $\{\fha^{-1}\}$ $C^k$-converges to $1/f$.
\end{proposition}

\begin{proof}
  Let us first show that $\fha^{-1}$ $C^0$-converges to $1/f$;
  one calculates
  \begin{align*}
    \limainfty&\norm{\fha^{-1}-\Ta(1/f)}
    \leq \limainfty\norm{\fha^{-1}}\norm{\mid_{N_\alpha}-\fha\Ta(1/f)}\\
    &=\limainfty\norm{\fha^{-1}}\norm{\mid_{N_\alpha}-\fha\Ta(1/f)+\Ta(1)-\Ta(1)}\\
    &\leq\limainfty\norm{\fha^{-1}}\parab{\norm{\mid_{N_\alpha}-\Ta(1)}
      +\norm{\fha\Ta(1/f)-\Ta(1)}}\\
    &=0,
  \end{align*}
  since the matrix regularization is unital and $||\fha^{-1}||$ is
  assumed to be uniformly bounded. Let us now proceed by induction and
  assume that $\fha^{-1}$ is $C^l$-convergent ($0\leq l<k$). For arbitrary
  $h\in C^\infty(\Sigma)$ it holds that
  \begin{align*}
    [\fha^{-1},\Ta(h)] = -\fha^{-1}[\fha,\Ta(h)]\fha^{-1},
  \end{align*}
  and since $\fha$ is $C^k$-convergent, the above sequence is
  $C^l$-convergent by Proposition \ref{prop:arbitrarySequences} which
  implies that $\fha^{-1}$ is $C^{l+1}$-convergent. Hence, it follows
  by induction that $\fha^{-1}$ is $C^k$-convergent. 
\end{proof}

\subsection{Discrete curvature and the Gauss-Bonnet theorem}\label{sec:discreteGB}

\noindent Let us now consider a surface $\Sigma$ embedded in
$M$ via the embedding coordinates $x^1,\ldots,x^m$, with a symplectic form
\begin{align*}
  \omega = \rho(u^1,u^2)du^1\wedge du^2,
\end{align*}
inducing the Poisson bracket
$\{f,h\}=\frac{1}{\rho}\eps^{ab}(\d_af)(\d_b h)$, and let $(\Ta,\ha)$
be a matrix regularization of $(\Sigma,\omega)$. Furthermore, we let
$\{\gammah_\alpha\}$ be a $C^2$-convergent sequence converging to
$\gamma=\sqrt{g}/\rho$ (and we assume that $\{\gammah_\alpha^{-1}\}$
exists and converges to $1/\gamma$), and we set $\Xa^i=\Ta(x^i)$ as well
as $\NAa^i=\Ta(n_A^i)$ for $i=1,\ldots,m$. Moreover, given the metric
$\gb_{ij}$ and the Christoffel symbols $\Gammab^{i}_{jk}$ of $M$, we
let $\{\Gh_{ij,\alpha}\}$ and $\{\Gammah^{i}_{jk,\alpha}\}$ denote
sequences converging to $\gb_{ij}$ and $\Gamma^{i}_{jk}$
respectively. To avoid excess of notation, we shall often suppress the
index $\alpha$ whenever all matrices are considered at a fixed (but
arbitrary) $\alpha$.

Since most formulas in Section \ref{sec:nambuPoissonFormulation} are
expressed in terms of the tensors $\P^i_j$ and $(\SA)^i_j$ (in the
case of surfaces), we introduce their matrix analogues
\begin{align*}
  \Ph^i_j &= \frac{1}{i\hbar}[X^i,X^{j'}]\Gh_{j'j}\\
  (\ShA)^i_j &= \frac{1}{i\hbar}[X^i,N_A^{j'}]\Gh_{j'j}+
  \frac{1}{i\hbar}[X^j,X^k]\Gammah^{j'}_{kl}N_A^l\Gh_{j'j},
\end{align*}
as well as their squares
\begin{align*}
  (\Ph^2)^i_j = (\Ph^i_k)^\dagger\Ph^k_j
  \quad\text{and}\quad
  (\ShA^2)^i_j=(\ShA{}^i_k)^\dagger\ShA{}^k_j,
\end{align*}
and corresponding trace
\begin{align*}
  \trh\Ph^2 =\sum_{i=1}^m(\Ph^2)^i_i
  \quad\text{and}\quad
  \trh\ShA^2 =\sum_{i=1}^m(\ShA^2)^i_i.
\end{align*}
(The ordinary trace of a matrix $X$ will be denoted by $\Tr X$.)  From
Proposition \ref{prop:arbitrarySequences} it follows that one can
easily construct matrix sequences converging to the geometric objects
in Section \ref{sec:nambuPoissonFormulation}, as long as the
appropriate type of convergence is assumed. Let us illustrate this by
investigating matrix sequences related to the curvature of $\Sigma$
and the Gauss-Bonnet theorem.

\begin{definition}
  Let $(\Ta,\hbar)$ be a matrix regularization of $(\Sigma,\omega)$,
  let $K$ be the Gaussian curvature of $\Sigma$ and let $\chi$ be the
  Euler characteristic of $\Sigma$.. A \emph{Discrete Curvature of
    $\Sigma$} is a matrix sequence $\{\Kh_1,\Kh_2,\Kh_3,\ldots\}$
  converging to $K$, and a \emph{Discrete Euler Characteristic of
    $\Sigma$} is a sequence $\{\chih_1,\chih_2,\chih_3,\ldots\}$ such
  that $\displaystyle\limainfty\chih_\alpha=\chi$.
\end{definition}

\noindent From the classical Gauss-Bonnet theorem, it is immediate to derive a
discrete analogue for matrix regularizations.

\begin{theorem}\label{thm:discreteEuler}
  Let $(\Ta,\hbar)$ be a matrix regularization of $(\Sigma,\omega)$, and let
  $\{\Kh_1,\Kh_2,\ldots\}$ be a discrete curvature of $\Sigma$. Then the sequence
  $\chih_1,\chih_2,\ldots$ defined by
  \begin{align}
    \chih_{\alpha} = \hbara\Tr\bracketb{\gammah_\alpha\Kh_{\alpha}},
  \end{align}
  is a discrete Euler characteristic of $\Sigma$.
\end{theorem}

\begin{proof}
  To prove the statement, we compute $\limainfty \chih_\alpha$ and show that it is equal to $\chi(\Sigma)$. Thus
  \begin{align*}
    \limainfty\chih_\alpha &= \limainfty\frac{1}{2\pi}2\pi\hbara\Tr\bracketb{\gammah_\alpha\Kh_{\alpha}},
  \end{align*}
  and by using Proposition \ref{prop:arbitrarySequences} we can write
  \begin{align*}
    \limainfty\chih_\alpha &= \frac{1}{2\pi}\int_{\Sigma}K\frac{\sqrt{g}}{\rho}\omega
    =\frac{1}{2\pi}\int_{\Sigma}K\frac{\sqrt{g}}{\rho}\rho dudv =
    \frac{1}{2\pi}\int_{\Sigma}K\sqrt{g}dudv = \chi(\Sigma),
  \end{align*}
  where the last equality is the classical Gauss-Bonnet theorem.
\end{proof}

\begin{theorem}\label{thm:discreteCurvature}
  Let $(\Ta,\hbar)$ be a unital matrix regularization of
  $(\Sigma,\omega)$ and let $\hat{R}_{ijkl}$, for each
  $i,j,k,l=1,\ldots,m$, be a sequence converging to the component of
  the curvature tensor of $M$. Then the sequence $\Kh$
  defined by
  \begin{align*}
    \Kh = \gammah^{-4}(\Ph^2)^{ik}(\Ph^2)^{jl}\hat{R}_{ijkl}
    -\frac{1}{2}\sum_{A=1}^p\paraa{\gammah^\dagger}^{-1}\paraa{\trh\ShA^2}\gammah^{-1},
  \end{align*}
  is a discrete curvature of $\Sigma$. Thus, a discrete Euler
  characteristic is given by
  \begin{align}
    \chih = \hbar\Tr\paraa{\gammah^{-3}(\Ph^2)^{ik}(\Ph^2)^{jl}\hat{R}_{ijkl}}
    -\frac{\hbar}{2}\sum_{A=1}^p\Tr\bracketb{\gammah^{-1}\trh\ShA^2}.
  \end{align}
\end{theorem}

\begin{proof}
  By using the way of constructing matrix sequences given through
  Proposition \ref{prop:arbitrarySequences}, the result follows
  immediately from Theorem \ref{thm:ricciCurvature}.
\end{proof}

\noindent In the case $M=\reals^m$ it follows from the results in
Section \ref{sec:surfaces} that when $(\Ta,\hbar)$ is a
$C^2$-convergent matrix regularization, then the sequence
\begin{equation}
  \begin{split}
    \Kh_\alpha=\frac{1}{\ha^4}\sum_{j,k,l=1}^m\Bigg(\frac{1}{2}&\paraa{\gammah_\alpha^\dagger}^{-2}
    \Ccom{\Xa^j}{\Xa^k}{\Xa^k}\Ccom{\Xa^j}{\Xa^l}{\Xa^l}\gammah_\alpha^{-2}\\
    &-\frac{1}{4}\paraa{\gammah_\alpha^\dagger}^{-2}\Ccom{\Xa^j}{\Xa^k}{\Xa^l}
    \Ccom{\Xa^j}{\Xa^k}{\Xa^l}\gammah_\alpha^{-2}\Bigg).
  \end{split}
\end{equation}
converges to the Gaussian curvature of $\Sigma$.

\subsection{Two simple examples}\label{sec:simpleExamples}

\subsubsection{The round fuzzy sphere}\label{sec:fuzzySphere}

\noindent For the sphere embedded in $\reals^3$ as 
\begin{align}
  \xv = (x^1,x^2,x^3) = (\cos\vphi\sin\theta,\sin\vphi\sin\theta,\cos\theta)
\end{align}
with the induced metric
\begin{align}
  (g_{ab}) = 
  \begin{pmatrix}
    1 & 0 \\ 0 & \sin^2\theta
  \end{pmatrix},
\end{align}
it is well known that one can construct a matrix regularization from
representations of $su(2)$. Namely, let $S_1,S_2,S_3$ be hermitian
$N\times N$ matrices such that $[S^j,S^k] = i{\epsilon^{jk}}_lS^l$,
$(S^1)^2+(S^2)^2+(S^3)^2=(N^2-1)/4$, and define
\begin{align}
  X^i = \frac{2}{\sqrt{N^2-1}}S^i.
\end{align}
Then there exists a map $\TN$ (which can be defined through expansion
in spherical harmonics) such that $\TN(x^i)=X^i$ and
$(\TN,\hbar=2/\sqrt{N^2-1})$ is a unital matrix regularization of
$(S^2,\sqrt{g}d\theta\wedge d\vphi)$ \cite{h:phdthesis}.  A unit normal of the sphere in
$\reals^3$ is given by $N\in T\reals^3$ with $N=x^i\d_i$, which gives
$N^i=X^i$, and one can compute the discrete curvature as
\begin{align}
  \Kh_N = -\frac{1}{\hbar^2}\sum_{i<j=1}^m\Tr[X^i,X^j]^2 = \mid_N 
\end{align}
which gives the discrete Euler characteristic
\begin{align}
  \chih_N &= \hbar\Tr\Kh_N = \hbar N = \frac{2N}{\sqrt{N^2-1}},
\end{align}
converging to $2$ as $N\to\infty$.

\subsubsection{The fuzzy Clifford torus}\label{sec:fuzzyTorus}

\noindent The Clifford torus in $S^3$ can be regarded as embedded in $\reals^4$ through
\begin{align*}
  \xv = (x^1,x^2,x^3,x^4) = \frac{1}{\sqrt{2}}(\cos\vphi_1,\sin\vphi_1,\cos\vphi_2,\sin\vphi_2),
\end{align*}
with the induced metric
\begin{align*}
  (g_{ab}) = \frac{1}{2}
  \begin{pmatrix}
    1 & 0 \\ 0 & 1
  \end{pmatrix},
\end{align*}
and two orthonormal vectors, normal to the tangent plane of the
surface in $T\reals^4$, can be written as
\begin{align*}
  N_\pm = x^1\d_1 + x^2\d_2 \pm x^3\d_3  \pm x^4\d_4.
\end{align*}
To construct a matrix regularization for the Clifford torus, one
considers the $N\times N$ matrices $g$ and $h$ with non-zero elements
\begin{align*}
  &g_{kk} = \omega^{k-1}\quad\text{ for $k=1,\ldots,N$}\\
  &h_{k,k+1} = 1\quad\text{ for $k=1,\ldots,N-1$}\\
  &h_{N,1} = 1,
\end{align*}
where $\omega=\exp(i2\theta)$ and $\theta=\pi/N$. These matrices satisfy
the relation $hg=\omega gh$.  The map $\TN$ is then defined on the
Fourier modes
\begin{align*}
  Y_{\mv}=e^{i\mv\cdot\vphiv}=e^{im_1\vphi_1+im_2\vphi_2}
\end{align*}
as  
\begin{align*}
  \TN(Y_{\mv}) = \omega^{\frac{1}{2}m_1m_2}g^{m_1}h^{m_2},
\end{align*}
and the pair $(\TN,\hbar=\sin\theta)$ is a unital matrix
regularization of the Clifford torus with respect to
$\sqrt{g}d\vphi_1\wedge d\vphi_2$
\cite{ffz:trigonometric,h:diffeomorphism}. Thus, using this map one
finds that
\begin{align*}
  &X^1 = T(x^1) = \frac{1}{\sqrt{2}}T(\cos\vphi_1) = \frac{1}{2\sqrt{2}}(\gd+g)\\
  &X^2 = T(x^2) = \frac{1}{\sqrt{2}}T(\sin\vphi_1) = \frac{i}{2\sqrt{2}}(\gd-g)\\
  &X^3 = T(x^3) = \frac{1}{\sqrt{2}}T(\cos\vphi_2) = \frac{1}{2\sqrt{2}}(\hd+h)\\
  &X^4 = T(x^4) = \frac{1}{\sqrt{2}}T(\sin\vphi_2) = \frac{i}{2\sqrt{2}}(\hd-h)
\end{align*}
which implies that $N_\pm^1=X^1$, $N_\pm^2=X^2$, $N_\pm^3=\pm X^3$ and
$N_\pm^4=\pm X^4$. By a straightforward computation one obtains
\begin{align*}
  -\frac{1}{\hbar^2}\sum_{i,j=1}^4[X^i,X^j]^2 = 2\mid
\end{align*}
and therefore
\begin{align*}
  \frac{1}{2\hbar^2}\sum_{i,j=1}^4[X^i,N^j_+][X^j,N^i_+]=-\frac{1}{2\hbar^2}\sum_{i,j=1}^4[X^i,X^j]^2 = \mid,
\end{align*}
and since $[X^1,X^2]=[X^3,X^4]=0$ it follows that
\begin{align*}
  \frac{1}{2\hbar^2}\sum_{i,j=1}^4[X^i,N^j_-][X^j,N^i_-]
  =\frac{1}{2\hbar^2}\sum_{i,j=1}^4[X^i,X^j]^2 = -\mid.
\end{align*}
This implies that the discrete curvature vanishes, i.e.
\begin{align*}
  \Kh_N = \frac{1}{2\hbar^2}\sum_{i,j=1}^4[X^i,N^j_+][X^j,N^i_+]
  +\frac{1}{2\hbar^2}\sum_{i,j=1}^4[X^i,N^j_-][X^j,N^i_-] = \mid-\mid = 0,
\end{align*}
which immediately gives $\chih_N=0$.

The following two examples will show that even in the smooth matrix
regularization of the torus it is easy to find sequences that are not
smooth, and that the regularization can be deformed into a non-smooth
matrix regularization.

\begin{example}\label{ex:nonSmoothSequence}
  Let $(\Ta,\ha)$ be the matrix regularization of the Clifford torus
  as in Section \ref{sec:fuzzyTorus}. For each $N$, define the matrix
  \begin{align*}
    \thetah = \diag(\hbar^s,0,\ldots,0),
  \end{align*}
  for some fixed $0<s\leq 1$. Clearly, it holds that
  \begin{align*}
    \limainfty\norm{\thetah-\Ta(0)}=\limainfty\norm{\thetah}=0,
  \end{align*}
  i.e. $\thetah$ $C^0$-converges to $0$. Let us show that $\thetah$
  does not $C^1$-converge to $0$. If $\thetah$ $C^1$-converges to $0$,
  then it must hold that
  \begin{align*}
    \limainfty\norm{\frac{1}{i\hbar}[\thetah,\Ta(f)]-\Ta\paraa{\{0,f\}}}=
    \limainfty\norm{\frac{1}{i\hbar}[\thetah,\Ta(f)]}=0
  \end{align*}
  for all $f\in C^\infty(\Sigma)$.
 For $H=2\sqrt{2}T_{(N)}(x^3)=h+h^\dagger$ one
  computes the eigenvalues of $A=\frac{1}{i\hbar}[\thetah,H]$ to be
  \begin{align*}
    \lambda_1=i\sqrt{2}\hbar^{s-1}\quad
    \lambda_2=-i\sqrt{2}\hbar^{s-1}\quad
    \lambda_3=\cdots=\lambda_N=0.
  \end{align*}
  Hence, the norm of $A$ does \emph{not} tend to $0$, which implies
  that $\thetah$ is not $C^1$-convergent.
\end{example}

\begin{example}\label{ex:deformedFuzzyTorus}
  Let $(\Ta,\ha)$ be the matrix regularization of the Clifford torus
  as in Section \ref{sec:fuzzyTorus}. For each $N$, define the matrix
  \begin{align*}
    \thetah = \diag(\hbar^s,0,\ldots,0),
  \end{align*}
  for some fixed $1<s\leq 2$.  Let us now deform the fuzzy torus to
  obtain a $C^1$-convergent matrix regularization that is not
  $C^2$-convergent. Defining
  \begin{align*}
    \Sa(f) = \Ta(f) + \mu(f)\thetah,
  \end{align*}
  where $\mu:C^\infty(\Sigma)\to\reals$ is an arbitrary linear
  functional, one can readily check that $(\Sa,\ha)$ is a $C^1$-convergent
  matrix regularization of the Clifford torus. Let us now prove that
  $(\Sa,\ha)$ is not a $C^2$-convergent matrix regularization, and let us for definiteness
  choose $\mu$ to be the evaluation map at $\vphi_1=\vphi_2=0$.

  In a $C^2$-convergent matrix regularization it holds that
  \begin{align*}
    \limainfty\norm{-\frac{1}{\hbar^2}\Ccom{\Sa(u)}{\Sa(v)}{\Sa(w)}
    -\Sa\paraa{\{\{u,v\},w\}}}=0,
  \end{align*}
  for all $u,v,w\in C^\infty(\Sigma)$. Choosing
  $u=2\sqrt{2}\cos\vphi_2$ and $v=w=2\sqrt{2}\sin\vphi_2$ gives
  $\Sa(u)=h^\dagger+h+2\sqrt{2}\thetah$, $\Sa(v)=i(h^\dagger-h)$ and
  $\{u,v\}=0$. Thus
  \begin{align*}
    \limainfty&\norm{-\frac{1}{\hbar^2}\Ccom{\Sa(u)}{\Sa(v)}{\Sa(w)}
    -\Sa\paraa{\{\{u,v\},w\}}}\\
  &=\limainfty\frac{2\sqrt{2}}{\hbar^2}\norm{\Ccom{\thetah}{i(h^\dagger-h)}{i(h^\dagger-h)}}
  =\limainfty 2\sqrt{2}\paraa{2+\sqrt{6}}\hbar^{s-2},
  \end{align*}  
  which does not converge to $0$. Hence, $(\Sa,\ha)$ is a
  $C^1$-convergent, but not $C^2$-convergent, matrix regularization of the Clifford torus.
\end{example}

\subsection{Axially symmetric surfaces in $\reals^3$}

\noindent Recall the classical description of general axially symmetric surfaces:
\begin{align}\label{axiallyuvparam}
  \xv &= \paraa{f(u)\cos v, f(u)\sin v, h(u)}\\
  \nv &= \frac{\pm 1}{\sqrt{h'(u)^2+f'(u)^2}}
  \paraa{h'(u)\cos v,h'(u)\sin v,-f'(u)}\notag,
\end{align}
which implies
\begin{align*}
  \paraa{g_{ab}}=
  \begin{pmatrix}
    f'^2+h'^2 & 0 \\
    0 & f^2
  \end{pmatrix}\qquad
  \paraa{h_{ab}} =\frac{\pm 1}{\sqrt{h'^2+f'^2}}
  \begin{pmatrix}
    h'f''-h''f' & 0\\
    0 & -fh'
  \end{pmatrix},
\end{align*}
where $h_{ab}$ are the components of the second fundamental form. The
Euler characteristic can be computed as
\begin{align}
  \chi = \frac{1}{2\pi}\int K\sqrt{g} = 
  -\int_{u_-}^{u_+}\frac{h'\paraa{h'f''-h''f'}}{\paraa{f'^2+h'^2}^{3/2}}du
  =-\frac{f'}{\sqrt{f'^2+h'^2}}\Bigg|_{u_-}^{u_+},
\end{align}
which is equal to zero for tori (due to periodicity) and equal to $+2$ for spherical surfaces ($f'(u_{\pm})=\mp\infty$ if $u=h$).

While a general procedure for constructing matrix analogues of
surfaces embedded in $\reals^3$ was obtained in
\cite{abhhs:noncommutative,abhhs:fuzzy} (cp. also \cite{a:repcalg}),
let us restrict now to $h(u)=u=z$, hence describe the axially
symmetric surface $\Sigma$ as a level set, $C=0$, of
\begin{align}
  C(\xv) = \frac{1}{2}\paraa{x^2+y^2-f^2(z)},
\end{align}
to carry out the construction in detail, and make the resulting
formulas explicit. Defining 
\begin{align}
  \{F(\xv),G(\xv)\}_{\reals^3} = \nabla C\cdot\paraa{\nabla F\times\nabla G},
\end{align}
one has 
\begin{align}
  \{x,y\}=-\ff'(z),\quad\{y,z\}=x,\quad\{z,x\} = y,
\end{align}
respectively
\begin{align}\label{eq:XYZCommutators}
  [X,Y] = i\hbar \ff'(Z),\quad [Y,Z]=i\hbar X,\quad [Z,X]=i\hbar Y
\end{align}
for the ``quantized'' (``non-commutative'') surface. In terms of the parametrization given in 
(\ref{axiallyuvparam}), the above Poisson bracket is equivalent to
\begin{align}
  \{F(u,v),G(u,v)\} = \eps^{ab}\paraa{\d_aF}\paraa{\d_b{G}}
\end{align}
where $\d_1=\d_v$ and $\d_2=\d_u$. By finding matrices of increasing
dimension satisfying (\ref{eq:XYZCommutators}), one can construct a
map $\Ta$ having the properties (\ref{eq:matrixProduct}) and
(\ref{eq:matrixCommutator}) of a matrix regularization restricted to
polynomial functions in $x,y,z$ (cp. \cite{a:phdthesis}).

For the round 2-sphere, $f(z)=1-z^2$, (\ref{eq:XYZCommutators}) gives
the Lie algebra $su(2)$, and its celebrated irreducible
representations satisfy
\begin{align}\label{eq:su2sumsquare}
  X^2+Y^2+Z^2 = \mid\quad\text{if}\quad \hbar=\frac{2}{\sqrt{N^2-1}}.
\end{align}
When $f$ is arbitrary, one can still find finite dimensional
representations of (\ref{eq:XYZCommutators}) as follows: rewrite
(\ref{eq:XYZCommutators}) as
\begin{align}
  &[Z,W] = \hbar W\label{eq:ZWCommutator}\\
  &[W,\Wd] = -2\hbar\ff'(Z)
\end{align}
implying that $z_i-z_j=\hbar$ whenever $W_{ij}\neq 0$ and $Z$
diagonal. Assuming $W=X+iY$ with non-zero matrix elements
$W_{k,k+1}=w_k$ for $k=1,\ldots,N-1$, one thus obtains (with
$w_0=w_N=0$)
\begin{align*}
  &Z_{kk} = \frac{\hbar}{2}\paraa{N+1-2k}\\
  &w_k^2-w_{k-1}^2=-2\hbar\ff'\paraa{\hbar(N+1-2k)/2}\equiv Q_k,
\end{align*}
which implies that
\begin{align*}
  w_ k^2 = \sum_{l=1}^kQ_l
\end{align*}
and the only non-trivial problem is to find the analogue of
(\ref{eq:su2sumsquare}). To this end, define 
\begin{align}\label{eq:fhdef}
  \fh^2 = X^2+Y^2 = \frac{1}{2}\paraa{W\Wd+\Wd W},
\end{align}
with $W$ given as above. As $Z$ has pairwise different eigenvalues,
the diagonal matrix given in (\ref{eq:fhdef}) can be thought of as a
function of $Z$; hence as $\fh^2(Z)$. It then trivially holds that
\begin{align}
  \Ch = X^2+Y^2-\fh^2(Z)=0,
\end{align}
for the representation defined above. The quantization of $\hbar$
comes through the requirement that $\fh^2$ should correspond to
$f^2$. While for the \emph{round} 2-sphere $\fh^2$ equals $f^2$,
provided $\hbar$ is chosen as in (\ref{eq:su2sumsquare}), it is easy
to see that in general they can not coincide, as
\begin{align*}
  [X^2+Y^2-&f(Z)^2,W] = [(W\Wd+\Wd W)/2-f(Z)^2,W]\\
  &=\frac{1}{2}W[\Wd,W]+\frac{1}{2}[\Wd,W]W-f(Z)[f(Z),W]-[f(Z),W]f(Z)\\
  &=\cdots=f(Z)\paraa{\hbar f'(Z)W-[f(Z),W]}+\paraa{\hbar f'(Z)W-[f(Z),W]}f(Z)
\end{align*}
with off-diagonal elements
\begin{align*}
  \paraa{f(z_k)+f(z_{k-1})}\paraa{\hbar f'(z_k)-(f(z_k)-f(z_{k-1}))}
\end{align*}
that are in general non-zero (hence $X^2+Y^2+f^2(Z)$ is usually not
even a Casimir, except in leading order).

How it \emph{does} work is perhaps best illustrated by a non-trivial example, $f(z)=1-z^4$:
\begin{align}
    w_k^2 =\frac{\hbar^4}{2}&\parab{(N+1)^3k-3(N+1)^2k(k+1)+\label{eq:wk}\\
      &2(N+1)k(k+1)(2k+1)-2k^2(k+1)^2}\notag\\
  \fh_k^2 = \frac{1}{2}(w_k^2&+w^2_{k-1}) = 
  \frac{\hbar^4}{4}\parab{(N+1)^3(2k-1)-6(N+1)^2k^2\notag\\ &\qquad\qquad+4(N+1)k(2k^2+1)-4k^2(k^2+1)}\notag
\end{align}
(note that $w^2_0=w_N^2=0$ is explicit in (\ref{eq:wk})) so that
\begin{align}
  \paraa{X^2+Y^2+Z^4}_{kk} = \hbar^4\bracketc{\frac{(N+1)^4}{16}-\frac{(N+1)^3}{4}+k(N+1)-k^2}.
\end{align}
Expressing the last two terms via $Z^2$ (note that the cancellation of
$k^3$ and $k^4$ terms shows the absence of $Z^3$ and higher
corrections) one finds
\begin{align*}
  X^2+Y^2+Z^4+\hbar^2Z^2 &= \hbar^4\frac{(N+1)^2}{16}\parab{(N+1)^2-4(N+1)+4}\mid\\
  &=\hbar^4\frac{(N^2-1)^2}{16}\mid,
\end{align*}
which equals $\mid$ if $\hbar$ is chosen as $2/\sqrt{N^2-1}$. Note
that this is the \emph{same} expression for $\hbar$ then for the round
sphere, $f^2=1-z^2$ (cp. (\ref{eq:su2sumsquare})).

A more elegant way to derive the quantum Casimir (cp. also \cite{r:repnonlinear,gps:beyondfuzzy})
\begin{align}
  Q = X^2+Y^2+Z^4+\hbar^2Z^2
\end{align}
is to calculate
\begin{align*}
  [X^2+Y^2+Z^4,W] &= [(W\Wd+\Wd W)/2+Z^4,W]\\
  &= \cdots = \hbar^2[W,Z^2],
\end{align*}
which determines the terms proportional to $\hbar$ in the Casimir. 

Due to the general formula 
\begin{align}
  \Kh = -\frac{1}{8\hbar^4}\eps_{jkl}\eps_{ipq}(\gammah^\dagger)^{-2}\coma{X^i,[X^k,X^l]}\coma{X^j,[X^p,X^q]}\gammah^{-2}
\end{align}
one obtains, for the axially symmetric surfaces discussed above,
\begin{align}
  \Kh = \gammah^{-2}\parac{(\ff')^2(Z)+\frac{1}{2\hbar}[W,\ff'(Z)]\Wd+\frac{1}{2\hbar}\Wd[W,\ff'(Z)]}\gammah^{-2}
\end{align}
with
\begin{align}
  \gammah^2 = \frac{1}{2}\paraa{W\Wd+\Wd W}+(\ff')^2(Z)
  =f(Z)^2\paraa{f'(Z)^2+\mid} + O(\hbar),
\end{align}
giving
\begin{align}
  &\Kh  = -\paraa{f'(Z)^2+\mid}^{-2}f(Z)^{-1}f''(Z) + O(\hbar)
\end{align}
and for $f(z)^2=1-z^4$ one has 
\begin{align}
  &\Kh = \paraa{4Z^6+\mid-Z^4}^{-2}\paraa{6Z^2-2Z^6}+O(\hbar)\\
  &\gammah^2 = \mid-Z^4+4Z^6+O(\hbar).
\end{align}
Note that (cp. (\ref{eq:ZWCommutator}))
$z_j-z_{j-1}=\hbar$ for arbitrary $f$, and that (due to the axial
symmetry) $\Kh$ and $\gammah^2$ are \emph{diagonal} matrices, so that
\begin{align*}
  \chih = \hbar\Tr\paraa{\sqrt{\gammah^2}\Kh},
\end{align*}
in this case simply being a Riemann sum approximation of $\int
K\sqrt{g}$, indeed converges to 2, the Euler characteristic of
spherical surfaces.

\subsection{A bound on the eigenvalues of the matrix Laplacian}\label{sec:laplaceBound}

\noindent As we have shown, many of the objects in differential
geometry can be expressed in terms of Nambu brackets. Let us now
illustrate, in the case of surfaces, that some of the techniques used
to prove classical theorems can be implemented for matrix
regularizations. In particular, let us prove that a lower bound on the
discrete Gaussian curvature induces a lower bound for the eigenvalues
of the discrete Laplacian. For simplicity, we shall consider the case
when $M=\reals^m$ and, in the following, all repeated indices are
assumed to be summed over the range $1,\ldots,m$.

Let us start by introducing the matrix analogue of the
operator $D^i$:
\begin{align*}
  \Dha^i(X) = \frac{1}{i\ha}\gammaha^{-1}[X,\Xa^i].
\end{align*}
These operators obey a rule of ``partial integration'', namely
\begin{align}\label{eq:discretePartialInt}
  \Tr\paraa{\gammaha\Dha^i(X)Y} = -\Tr\paraa{\gammaha\Dha^i(Y)X},
\end{align}
which is in analogy with the fact that
\begin{align*}
  \int_\Sigma \paraa{\gamma D^i(f)h}\omega = -\int_\Sigma\paraa{\gamma D^i(h)f}\omega.
\end{align*}
In view of Proposition \ref{prop:covderivFormulas}, it is natural to
make the following definition:
\begin{definition}
  Let $(\Ta,\ha)$ be a matrix regularization of $(\Sigma,\omega)$. The
  \emph{Discrete Laplacian on $\Sigma$} is a sequence
  $\{\Deltaha\}$ of linear maps defined as
  \begin{align*}
    \Deltaha(X) = \Dha^j\Dha^j(X) = 
    -\frac{1}{\ha^2}\gammaha^{-1}\big[\gammaha^{-1}[X,\Xa^j],\Xa^j\big],
  \end{align*}
  where $X$ is a $\Na\times\Na$ matrix. An \emph{eigenmatrix sequence
    of $\Deltaha$} is a convergent sequence $\{\uha\}$ such that
  $\Deltaha(\uha)=\la\uha$ for all $\alpha$ and
  $\displaystyle\limainfty\la=\lambda$.
\end{definition}

\begin{proposition}
  A $C^2$-convergent eigenmatrix sequence of $\Deltaha$ converges to
  an eigenfunction of $\Delta$ with eigenvalue
  $\lambda=\displaystyle\limainfty\la$.
\end{proposition}

\begin{proof}
  Given the assumption that $\uha$ is a $C^2$-convergent matrix
  sequence converging to $u$, we want to prove that $\Delta u-\lambda
  u=0$. By Proposition \ref{prop:uniqueNormZero} this is equivalent to
  proving that $\limainfty||\Ta(\Delta u-\lambda u)||=0$. One obtains
  \begin{align*}
    \limainfty&\norm{\Ta(\Delta u-\lambda u)} = 
    \limainfty\norm{\Ta(\Delta u)-\Deltaha\uha+\Deltaha\uha-\lambda \Ta(u)+\lambda\uha-\lambda\uha}\\
    &\leq\limainfty\parac{\norm{\Ta(\Delta u)-\Deltaha\uha}+|\lambda|\norm{-\Ta(u)+\uha}+\norm{\Deltaha\uha-\lambda\uha}}\\
    &=\limainfty\norm{\Deltaha\uha-\lambda\uha}
    \leq\limainfty\parab{\norm{\Deltaha\uha-\la\uha}+|\lambda-\la|\norm{\uha}}=0,
  \end{align*}
  since $\Deltaha\uha-\la\uha=0$ and $\la$ converges to $\lambda$.
\end{proof}

\noindent The way curvature is introduced in the classical proof of
the bound on the eigenvalues, is through the commutation of covariant
derivatives. Let us state the corresponding result for matrix
regularizations.

\begin{proposition}\label{prop:curvatureeq}
  Let $(\Ta,\ha)$ be a $C^2$-convergent matrix regularization of
  $(\Sigma,\omega)$. If $\{\uha\}$ is a $C^3$-convergent matrix
  sequence then
  \begin{align*}
    \limainfty&\Big|\Big|
      \Dha^i(\uha)\Dha^i\Dha^j\Dha^j(\uha)
      -\Dha^i(\uha)\Dha^j\Dha^j\Dha^i(\uha)\\
      &\qquad-\ldbrack\Dha^i,\Dha^j\rdbrack(\uha)\Dha^i\Dha^j(\uha)
       +\Kha\Dha^i(\uha)\Dha^i(\uha)\Big|\Big|=0,
  \end{align*}
  where $\ldbrack\cdot,\cdot\rdbrack$ denotes the commutator with
  respect to composition of maps.
\end{proposition}

\begin{proof}
  The result follows immediately from Proposition
  \ref{prop:covariantDRicci} and Proposition
  \ref{prop:arbitrarySequences}. Note that in the case of surfaces it
  holds that $\R_{ab}=Kg_{ab}$, where $K$ is the Gaussian curvature of
  $\Sigma$.
\end{proof}

\noindent A useful corollary is the following:

\begin{proposition}\label{prop:intCurvatureEq}
  Let $(\Ta,\ha)$ be a $C^2$-convergent matrix regularization of
  $(\Sigma,\omega)$. If $\{\uha\}$ is a $C^2$-convergent matrix sequence then
  \begin{align*}
    \limainfty&\ha\Tr\parab{\gammaha\Dha^i\Dha^j\Dha^j(\uha)\Dha^i(\uha)}=\\
    % &\limainfty\ha\Tr\parab{-\Dha^i\Dha^j(\uha)\Dha^j\Dha^i(\uha)
    % -\Kha\Dha^i(\uha)\Dha^i(\uha)}\\
  &\limainfty\ha\Tr\parab{\gammaha\Dha^j\Dha^i\Dha^j(\uha)\Dha^i(\uha)
    -\gammaha\Kha\Dha^i(\uha)\Dha^i(\uha)}
  \end{align*}  
\end{proposition}

\begin{proof}
  It follows from Proposition \ref{prop:curvatureeq} that for a
  $C^3$-convergent sequence $\uha$ it holds that
  \begin{align*}
      &\limainfty\hbara\Tr\Big(\gammaha\Dha^i\Dha^j\Dha^j(\uha)\Dha^i(\uha)
      -\gammaha\Dha^j\Dha^j\Dha^i(\uha)\Dha^i(\uha)\\
      &\qquad-\gammaha\ldbrack\Dha^i,\Dha^j\rdbrack(\uha)\Dha^i\Dha^j(\uha)
       +\gammaha\Kha\Dha^i(\uha)\Dha^i(\uha)\Big)=0.   
  \end{align*}
  Due to the appearance of a trace, the above holds even for
  $C^2$-convergent sequences, since e.g.
  \begin{align*}
    \hbara\Tr\gammaha\Dha^i\Dha^j\Dha^j(\uha)\Dha^i(\uha)
    =-\hbara\Tr\gammaha\Dha^i\Dha^i(\uha)\Dha^j\Dha^j(\uha),
  \end{align*}
  and the latter expression only requires $C^2$-convergence. Thus, one obtains
  \begin{align*}
    \limainfty&\ha\Tr\parab{\gammaha\Dha^i\Dha^j\Dha^j(\uha)\Dha^i(\uha)}
    =\limainfty\ha\Tr\Big(\gammaha\Dha^j\Dha^j\Dha^i(\uha)\Dha^i(\uha)\\
    &+\gammaha\ldbrack\Dha^i,\Dha^j\rdbrack(\uha)\Dha^i\Dha^j(\uha)
       -\gammaha\Kha\Dha^i(\uha)\Dha^i(\uha)\Big)\\
       &=\limainfty\ha\Tr\parab{\gammaha\Dha^j\Dha^i\Dha^j(\uha)\Dha^i(\uha)
    -\gammaha\Kha\Dha^i(\uha)\Dha^i(\uha)},
  \end{align*}
  by using equation (\ref{eq:discretePartialInt}).
\end{proof}

\begin{proposition}\label{prop:laplaceDtwoineq}
  Let $(\Ta,\ha)$ be a matrix regularization of
  $(\Sigma,\omega)$. If $\{\uha\}$ is a $C^2$-convergent matrix sequence then
  \begin{align*}
    \limainfty \ha\Tr\parab{\Dha^i\Dha^j(\uha)\Dha^j\Dha^i(\uha)}\geq
    \frac{1}{2}\limainfty \ha\Tr\paraa{\Deltaha(\uha)}^2.
  \end{align*}
\end{proposition}

\begin{proof}
  By using the fact that $|\nabla^2u|^2\geq \frac{1}{2}(\Delta u)^2$
  (for 2-dimensional manifolds) one obtains
  \begin{align*}
    \limainfty&\ha\Tr\parab{\Dha^i\Dha^j(\uha)\Dha^j\Dha^i(\uha)}
    =\frac{1}{2\pi}\int_\Sigma|\nabla^2u|^2\omega
    \geq\frac{1}{4\pi}\int_\Sigma(\Delta u)^2\omega\\
    &=\limainfty\frac{1}{2}\ha\Tr\paraa{\Deltaha(\uha)}^2,
  \end{align*}
  since $\uha$ is assumed to $C^2$-converge to $u$.
\end{proof}

\begin{theorem}
  Let $(\Ta,\ha)$ be a $C^2$-convergent matrix regularization of
  $(\Sigma,\omega)$ and let $\{\uha\}$ be a $C^2$-convergent
  eigenmatrix sequence of $\Deltaha$ with eigenvalues $\{-\la\}$. If
  $\Kha\geq\kappa\mid_{\Na}$ for some $\kappa\in\reals$ and all
  $\alpha>\alpha_0$, then $\displaystyle\limainfty\lambda_\alpha\geq
  2\kappa$.
\end{theorem}

\begin{proof}
  Let $\{\uha\}$ be a hermitian eigenmatrix sequence of $\Deltaha$ with
  eigenvalues $\{-\la\}$. First, one rewrites
  \begin{equation}\label{eq:LasqDtwo}
    \begin{split}
      \Tr\gammaha\Deltaha(\uha)^2 &= \Tr\paraa{\gammaha\Dha^i\Dha^i(\uha)\Dha^j\Dha^j(\uha)}\\
      &= -\la\Tr\paraa{\uha\gammaha\Dha^i\Dha^i(\uha)}
      = \la\Tr\paraa{\gammaha\Dha^i(\uha)\Dha^i(\uha)}.
    \end{split}
  \end{equation}
  Then, one makes use of Proposition \ref{prop:intCurvatureEq} to write
  \begin{align*}
    \limai&\ha\Tr\gammaha\Deltaha(\uha)^2 = -\limai\ha\Tr\paraa{\gammaha\Dha^i\Dha^j\Dha^j(\uha)\Dha^i(\uha)}\\
    &=\limai\ha\Tr\Big(-\gammaha\Dha^j\Dha^i\Dha^j(\uha)\Dha^i(\uha)+\gammaha\Kha\Dha^i(\uha)\Dha^i(\uha)\Big)\\
    &=\limai\ha\Tr\parab{\gammaha\Dha^j\Dha^i(\uha)\Dha^i\Dha^j(\uha)+\gammaha\Kha\Dha^i(\uha)\Dha^i(\uha)}.
  \end{align*}
  Using the assumption that $\Kha\geq\kappa\mid$ together with
  Proposition \ref{prop:laplaceDtwoineq} one obtains
  \begin{align*}
    \limai\ha\Tr\gammaha\Deltaha(\uha)^2 &\geq
    \limai\ha\Tr\parac{\frac{1}{2}\gammaha\Deltaha(\uha)^2+\kappa\gammaha\Dha^i(\uha)\Dha^i(\uha)}\\
    &=\limai\parac{\frac{1}{2}\la+\kappa}\ha\Tr\paraa{\gammaha\Dha^i(\uha)\Dha^i(\uha)},
  \end{align*}
  where (\ref{eq:LasqDtwo}) has been used. One can now compare the
  above inequality with (\ref{eq:LasqDtwo}) to obtain
  \begin{align*}
    \frac{1}{2}(\lambda-2\kappa)\limai\ha\Tr\paraa{\gammaha\Dha^i(\uha)\Dha^i(\uha)}\geq 0.
  \end{align*}
  Since 
  \begin{align*}
    \limai\ha\Tr\paraa{\gammaha\Dha^i(\uha)\Dha^i(\uha)}
    =\frac{1}{2\pi}\int_\Sigma\gamma|\nabla u|^2\omega\geq 0,
  \end{align*}
  due to the fact that $\gamma$ is a positive function, it follows
  that $\lambda\geq 2\kappa$.
\end{proof}

\noindent Although the above proof depends on the fact that the matrix
regularization is associated to a surface (and therefore, the results
of differential geometry can be employed), we believe that, under
suitable conditions on the matrix algebra, there exists a proof that
is independent of this correspondence.

\section*{Acknowledgments}

\noindent J.A. would like to thank the Institut des Hautes \'Etudes
Scientifiques for hospitality and H. Shimada for discussions on matrix
regularizations,  while J.H. thanks M. Bordemann for many discussions
on related topics (and for switching talks at the October 2009 AEI
workshop ``Membranes, Minimal Surfaces and Matrix Limits'').

\bibliographystyle{alpha}
\bibliography{nambudiscrete}

\end{document}